\documentclass[12pt]{amsart}

\usepackage{parskip, amsthm,amsmath,amssymb,amscd,verbatim}

\usepackage{amsfonts, amssymb, amsmath, amsthm,  mathrsfs, stmaryrd, nopageno }
\usepackage[all]{xy}
\usepackage{verbatim}
\usepackage{tikz}
\usepackage{graphicx}
\usepackage{hyperref}

\usetikzlibrary{matrix,arrows,decorations.pathmorphing}
  \theoremstyle{plain}
  \newtheorem{theorem}{Theorem}[section]
  \newtheorem{lemma}[theorem]{Lemma}

  \newtheorem{cor}[theorem]{Corollary}
  
  \newtheorem{conj}[theorem]{Conjecture}

  \theoremstyle{remark}

  \theoremstyle{definition}

  \newtheorem{Remark}[theorem]{Remark}

  \numberwithin{equation}{section}
  \numberwithin{theorem}{section}

\newtheorem{proposition}[theorem]{Proposition}
\newtheorem{definition}[theorem]{Definition}

\newcommand{\Z}{\mathbb{Z}}

\newcommand{\C}{\mathbb{C}}
\newcommand{\Oh}{\mathcal{O}}

\newcommand{\Hz}{\mathcal{H}}

\newcommand{\HH}{\overline{\mathcal{H}}}
\newcommand{\FF}{\mathbb{F}}

\newcommand{\mf}{\mathfrak}
\newcommand{\mc}{\mathcal}

\newcommand{\B}{\mathcal{B}}

\newcommand{\ra}{\rightarrow}
\newcommand{\xra}{\xrightarrow}

\DeclareMathOperator{\diag}{diag}
\DeclareMathOperator{\rank}{rank}

\DeclareMathOperator{\rec}{rec}

\DeclareMathOperator{\tr}{tr}

\DeclareMathOperator{\im}{im}

\DeclareMathOperator{\GL}{GL}

\DeclareMathOperator{\id}{id}

\DeclareMathOperator{\Hom}{Hom}

\DeclareMathOperator{\res}{res}

\input xy
\xyoption{all}

\author{Jon Cohen}

\begin{document}

\title{Transfer of representations and orbital integrals for inner forms of $\GL_n$}


 \begin{abstract}
 We characterize the Local Langlands Correspondence (LLC) for inner forms of $\GL_n$ via the Jacquet-Langlands Correspondence (JLC) and compatibility with the Langlands Classification. We show that LLC satisfies a natural compatibility with parabolic induction and characterize LLC for inner forms as a unique family of bijections $\Pi(\GL_r(D)) \ra \Phi(\GL_r(D))$ for each $r$, (for a fixed $D$) satisfying certain properties. We construct a surjective map of Bernstein centers $\mf{Z}(\GL_n(F))\to \mf{Z}(\GL_r(D))$ and show this produces pairs of matching distributions in the sense of \cite{SBC}. Finally, we construct explicit Iwahori-biinvariant matching functions for unit elements in the parahoric Hecke algebras of $\GL_r(D)$, and thereby produce many explicit pairs of matching functions. 
  \end{abstract}

\maketitle
\thispagestyle{empty}
\pagestyle{headings}
\markboth{Jon Cohen}{Transfer of Representations and Orbital Integrals for Inner Forms of $\GL_n$}

\section{Introduction} \label{section intro}

The classical Local Langlands Correspondence for $G^*=\GL_n$ over a local nonarchimedian field $F$ is characterized as the unique family of bijections between irreducible smooth complex representations $\Pi(G^*)$ of $\GL_n(F)$ and $n$-dimensional Frobenius-semisimple Weil-Deligne representations $\Phi(G^*)$, satisfying certain properties. Of these properties, the most crucial is the preservation of $L$- and $\varepsilon$-factors of pairs, originally constructed in \cite{jacquet1983rankin}. The absence of an intrinsic definition of such factors is an obstacle to generalizing the characterization to inner forms $\GL_r(D)$, where $D$ is an $F$-central division algebra. The theory of Whittaker models, necessary to define $L(\pi\times\pi', s)$, exists only for quasi-split groups, and so does not apply to these inner forms. It is also not clear what form such factors could take, since the functional equation for $L(\pi\times \pi', s)$ uses the existence of inverse-transpose as an involution transporting irreducible representations of $\GL_n(F)$ to their contragredients. With the single exception of $\GL_r(D)$ as an inner form of $\GL_{2r}(F)$, i.e., when $D$ is a quaternion algebra, no such involution exists for $\GL_r(D)$, as all automorphisms of $\GL_r(D)$ are inner; see \cite{raghuram2002representations}. 



Instead, we characterize the Local Langlands Correspondence for inner forms of $G^*$ by requiring it to be compatible with the Jacquet-Langlands Correspondence and Langlands quotients. We also consider the restrictions of the resulting Langlands parameters to the Weil group. Let $\Pi(H)$, respectively $\Pi^2(H)$, denote the set of isomorphism classes of irreducible smooth complex $H$-representations, respectively essentially square-integrable representations. Let $\Phi(H)$ denote the set of equivalence classes of Langlands parameters $W_F'\to {^L} H$. The result is summarized in the following theorem. 

\begin{theorem}
Fix $D$ a $d^2$-dimensional $F$-central division algebra over the nonarchimedian local field $F$, and let $G=\GL_r(D)$ be an inner form of $G^*=\GL_n(F)=\GL_{rd}(F)$. Then there is a a unique family of bijective maps $$\rec_r:\Pi(\GL_r(D)) \to \Phi(\GL_r(D))\subset \Phi(\GL_n(F))$$ for $r\geq 1$ such that 

1. $\rec_r|_{\Pi^2(\GL_r(D))} = \rec^*_{rd}\circ JL_r$ where $\rec_n^*: \Pi(\GL_n(F))\xrightarrow{\cong} \Phi(\GL_n(F))$ is the LLC for $\GL_n(F)$ and $JL_r: \Pi^2(G)\to \Pi^2(G^*)$ is the Jacquet-Langlands correspondence.  

2. If $\pi$ is the Langlands quotient of $\sigma_1\times \ldots \times \sigma_k$ then $\rec_r(\pi) = \rec_{r_1}(\sigma_1)\oplus \cdots \oplus \rec_{r_k}(\sigma_k)$. 

The image of $\rec_r$ consists of those parameters which decompose as direct sums of indecomposable Weil-Deligne representations $W'_F \to \GL_m(\C)$ where $d$ divides $m$. The family $\rec_r$ is compatible with twisting: if $\chi\circ Nrd_{\GL_r(D)}$ is an arbitrary character of $\GL_r(D)$, where $\chi:F^\times \to \C^\times$, then $\rec_r({\chi\cdot\pi})=\rec_r(\pi) \otimes (\chi\circ Art_F^{-1})$ where $Art_F: F^\times \to W_F^{ab}$ is the Artin Reciprocity isomorphism of Local Class Field Theory. 

If we postcompose $\rec_r$ with restriction to the Weil group $W_F$, then the resulting family is  characterized by compatibility with Jacquet-Langlands and parabolic induction in the sense of LLC+ \textup{(}see section 3 for the definition of LLC+\textup{)}. The image of $\res|_{W_F}\circ \rec_r$ is those representations $W_F\to \GL_{rd}(\C)$ which factor through a Levi subgroup of the form $\prod\limits_{i=1}^k \GL_{m_i \frac{ d}{  q_i}}(\C)^{q_i}$ for some coprime integers $m_i$ and $q_i$ such that $\sum\limits_{i=1}^k m_i = r$. 
\end{theorem}

 This is Theorem \ref{characterizationtheorem} below. We mention that the LLC for inner forms of GL(n) has a folklore status and is considered well-known; see \cite{aubert2012geometric}. However, the
characterization given above does not seem to appear
elsewhere, and is useful for the other applications in this paper.

With this in hand we then prove a result relating the Bernstein centers $\mf{Z}(G)$ and $\mf{Z}(G^*)$ of $G$ and $G^*$. See section \ref{section morphismT} below for a recollection of the relevant basics of the theory of the Bernstein Center.

\begin{theorem} \label{theorem2}
There is a natural surjective homomorphism \begin{equation*}
\mf{T}:\mf{Z}(G^*)\to \mf{Z}(G)
\end{equation*} between the Bernstein center of $G^*$ and that of its inner form $G=\GL_r(D)$. Furthermore, if $f\in \Hz(G)$ and $f^*\in \Hz(G^*)$ have matching orbital integrals \textup{(}see section \ref{sectionmatching} for the relevant definitions\textup{)}, then so do $\mf{T}(Z^*) * f$ and $Z^**f^*$ for all $Z^*\in \mf{Z}(G^*)$.  
\end{theorem}

This is shown in Proposition \ref{5.1Prop}, Theorem \ref{5.1.Thm}, and Theorem \ref{MatchDist}. We show that the restriction of the Langlands parameter $\varphi_\pi= \rec_r(\pi)$ of an irreducible representation $\pi\in \Pi(G)$ to $W_F$ is determined by the supercuspidal support of $\pi$, and use this to prove a transference of geometric Bernstein centers between $G^*$ and its inner forms, as we now explain. Let $I_F$ denote the inertia subgroup of the Weil group, and let $\Phi$ denote a geometric Frobenius. The geometric Bernstein centers are defined to be the subalgebras of the Bernstein centers generated by those regular functions on the Bernstein variety of the form 
\begin{equation*}
Z_V:(M, \sigma)_G\mapsto \tr(\varphi_\pi(\Phi): V^{I_F})
\end{equation*} 
where $V$ varies over finite-dimensional algebraic representations of ${^L }G$, and $(M, \sigma)_G$ is the supercuspidal support of $\pi$. The fact that these $Z_V$ are well-defined elements of the Bernstein center is a consequence of LLC+, as we describe in section \ref{sectionapplication}. We write $Z_V^*$ for the analogous distribution on $G^*$. As an immediate consequence of the definition of $\mf{T}$ and LLC+ we obtain the following 

\begin{cor}\label{introcor}

For a representation $V$ of $^L G$ as above, we have $\mf{T}(Z_V^*) = Z_V$. 

\end{cor}

This result, listed below as Corollary \ref{Z_V}, had been previously proved in the special case of $n=2$ for the Iwahori block in \cite{SBC}. There is a concrete combinatorial interpretation of this transference, at least on the Iwahori block, in terms of an identity involving Schur polynomials; see section \ref{sectioncombasides}. 


In order to apply Theorem \ref{theorem2}, we need to be given a pair of matching functions. It is known that for every $f\in \Hz(\GL_r(D))$ there is a function $f^*\in \Hz(\GL_n(F))$ which matches $f$. However, this $f^*$ is not given explicitly. Indeed, only its image in the cocenter is canonical. However, having an explicit matching function, particularly one with specified biinvariance properties, can be useful for many purposes. For example, the matching of Kottwitz's Euler-Poincare functions is used in his proof of the Tamagawa Number Conjecture, \cite{Kott}. We briefly recall the definition of these functions. Let $H$ be a reductive, adjoint, $F$-group, and let $S$ denote some choice of representatives for the (finitely-many) $H$-orbits of facets in the Bruhat-Tits building of $H$. For $\sigma\in S$, let $H_\sigma\subset H$ denote the stabilizer, and let $sgn_\sigma: H_\sigma\to\{ \pm1\} $ denote the sign of the permutation action of $H_\sigma$ on the vertices of $\sigma$. Fix a Haar measure $dh$ on $H$. Then Kottwitz's Euler-Poincare function on $H$ is defined to be $$\sum\limits_{\sigma\in S} \frac{(-1)^{\dim\sigma} sgn_\sigma}{dh(H_\sigma)}.$$

Returning to our context of $G=\GL_r(D)$, let $J$ denote any parahoric subgroup of $G$. For each Levi subgroup $M^*\subset G^*$, we construct explicit functions $f_{S_{M^*}}$ in the Iwahori-Hecke algebra of $\GL_n(F)$, which we term ``relative Euler-Poincare functions.'' Explicitly, these have the form $$f_{S_{M^*}} := \left|\frac{M^*}{M^{*,1} Z(M^*)}\right| \sum\limits_{\sigma\in S_{M^*}}\frac{(-1)^{\dim\sigma} }{ vol(M^*_\sigma/ Z(M^*), dm/dz )}  \frac{vol(M^*_\sigma \cap M^{*,1}, dm)  }{vol(J_\sigma, dg)} \, 1_{J_\sigma}.$$  See section \ref{sectionmatchingfor1J} for complete definitions. In Theorem \ref{defofF_J} we define a certain linear combination $F_J$ of these $f_{S_{M^*}}$, depending on $J$. 
In the case that $G=\GL_1(D)$, so that $J=\Oh_D^\times$ is the Iwahori subgroup, then $F_J$ is Kottwitz's Euler-Poincare function for $PGL_n(F)$, lifted to $\GL_n(F)$ and restricted to $\GL_n(F)^1$.   Subject to some restrictions on $F$, we prove the following result. 

\begin{theorem}
Let $F$ have characteristic zero and residual characteristic $p$. Assume that $p>n$. Then $F_J$ has matching orbital integrals to $1_J$, the unit element of the parahoric Hecke algebra $\Hz(G, J)$ of $\GL_r(D)$.
\end{theorem}

By applying Theorem \ref{theorem2}, we obtain a large family of matching functions. In particular, with Corollary \ref{introcor}, we have the following immediate consequence.  

\begin{cor}
Under the same assumptions on $F$, the functions $Z_V * 1_J$ and $Z_V^* * F_J$ have matching orbital integrals. 
\end{cor} 

The construction of $f_{S_{M^*}}$ makes use only of Bruhat-Tits theory, and can be generalized to arbitrary reductive $p$-adic groups. See Definition \ref{gendef}. These functions may be of interest in greater generality than we pursue here. 


For future work, we will extend these results to twisted orbital integrals, as these will then have applications to Shimura varieties. 

We also prove a result relating the cocenters of $G$ and $G^*$, whose proof makes use of the homomorphism $\mf{T}$ of Bernstein centers. See section \ref{section cocenter}. 

\begin{theorem}
There is a natural surjective homomorphism \begin{equation*}
\overline{\mf{T}}:\HH(G^*)\to \HH(G). 
\end{equation*} 
\end{theorem}

\subsection*{Acknowledgments}

I would like to thank my advisor Thomas Haines for his help throughout the process of writing this paper. This research has been partially funded by NSF grants DMS-0901723 and DMS-1406787.

\section{Notation} \label{section notation}

Let $F$ be a local nonarchimedian field, and $D$ a $d^2$-dimensional $F$-central division algebra. We write $\Oh=\Oh_F$ for the valuation ring of $F$, $\varpi$ for a uniformizing parameter, and $q$ for the cardinality of the residue field. Similarly we will write $\Oh_D$ and $\varpi_D$ for the valuation ring and a choice of uniformizer for $D$, respectively. Any restrictions on the characteristic of $F$ or its residue field will be made in the section where they arise. For us $G$ will denote the group $\GL_r(D)$ and $G^*$ the group $\GL_n(F)$, where $n=rd$, unless otherwise specified. We will sometimes write $\varphi_\pi$ instead of $\rec_n(\pi)$ if we wish to suppress the index $n$. Given a representation $\sigma$ of a Levi subgroup $M\subset G$, we will write $\iota_P^G \sigma$ for the normalized parabolically induced representation of $G$ along some parabolic subgroup $P$ with Levi factor $M$. If $\sigma$ is a finite-length representation and we are working inside the Grothendieck group of such representations, we may sometimes write  $\iota_M^G \sigma$ for the corresponding semisimplification, as it does not depend on the parabolic subgroup $P$. We similarly write $r_P^G(\pi)$ or $\pi_N$ for the normalized Jacquet module with respect to a parabolic subgroup $P$ with unipotent radical $N$. We write $r_M^G(\pi)$ for its image in the Grothendieck group. We also write $X(G)$ for its group of unramified characters, and $G^1$ for the kernel of the Kottwitz homomorphism. In our case, $G^1$ is those elements of $G$ whose reduced norm is in $\Oh^\times$. Similar notations hold for $G^*$. All representations are complex and smooth. 

We denote by $\mf{X}_G$ the Bernstein variety of supercuspidal supports of $G$, and write $\mf{s}=[M, \sigma]_G$ for an inertial support. We write $\mf{X_s}$ for the corresponding Bernstein component. We write $\mf{Z}(G) = \C[\mf{X}_G]$ for the Bernstein center of $G$. The symbol $\Pi(G)$ denotes the set of all irreducible smooth representations of $G$, and $\Pi^2(G)$ denotes the subset of essentially square-integrable representations. We denote by $\Hz(G)$ the Hecke algebra of compactly-supported, complex-valued functions on $G$; the choice of Haar measure here will be made clear when it is relevant.  
 
 \section{The LLC and LLC+ for $\GL_n(F)$} \label{section LLCforGLn}
 
 The goal of this section is to state two compatibilities satisfied by the Local Langlands Correspondence for $G^*:=\GL_n(F)$. We recall the necessary aspects of the Bernstein-Zelevinsky classification of irreducible smooth representations of $G^*$ in terms of supercuspidals. We freely use the notation and terminology of \cite{zelevinsky1980induced}. Let 
 $$\nu=|\det|_F: G^*\to \C^\times.$$ We use the same symbol for all $n$. Let $\rho$ be a supercuspidal representation of $\GL_{\frac{n}{k} }(F)$ for some divisor $k$ of $n$. If $\Delta=[\rho, \nu^{k-1}\rho]$ is a segment of supercuspidal representations, we write $Q(\Delta)$ for the unique irreducible quotient of the normalized parabolically induced representation
 $$\rho\times \nu\rho \times \cdots \times \nu^{k-1}\rho := i_P^{G^*} (\boxtimes_{i=0}^{k-1} \nu^i \rho)$$ where $P$ is the block upper-triangular parabolic subgroup with Levi factor $\GL_{
\frac{n}{k} } (F)^k$ Then $Q(\Delta)$ is essentially square-integrable, and all such representations so arise for a unique $\Delta$. Let $\pi$ be an arbitrary irreducible smooth representation of $G^*$. By \cite{zelevinsky1980induced}, $\pi$ is the unique irreducible quotient of $i_{P'}^{G^*} (Q(\Delta_1)\boxtimes \ldots \boxtimes Q(\Delta_r))$ for some segments $\Delta_i$, subject only to the condition that $i<j$ implies $\Delta_i$ does not precede $\Delta_j$. We write $\pi=Q(\Delta_1, \ldots, \Delta_r)$.

\begin{lemma} \label{LangQuotLLC}
Let $\pi=Q(\Delta_1, \ldots, \Delta_r)$ as above, and let $\sigma=\boxtimes_{i=1}^r Q(\Delta_i)\in \Pi^2(M)$ where $M$ is the Levi subgroup $M=\prod\limits_{i=1}^r \GL_{n_i}(F)$. Then the Langlands parameters of $\pi$ and $\sigma$ are equal. That is, if
  \begin{equation*}
\varphi_\sigma:=\bigoplus\limits_{i=1}^r \varphi_{Q(\Delta_i)} :W_F' \ra {^L} M =\left( \prod\limits_{i=1}^r \GL_{n_i}(\C) \right)\rtimes W_F \subset \GL_n(\C)\rtimes W_F={^L } G 
  \end{equation*}
  corresponds to $\sigma$ and $\varphi_\pi$ corresponds to $\pi$, then they are conjugate by $\widehat{G}$. 
\end{lemma}

 \begin{proof} The result is tautological if $\pi$ is essentially square-integrable; in particular, if it is supercuspidal. So we assume it is not. Then by \cite[Theorem 1.7]{henniart2002caracterisation}, it is enough to show that 
 \begin{equation*}
L(s, \varphi_\pi\otimes \tau) = L(s, \varphi_\sigma\otimes \tau)
 \end{equation*}for every $\tau\in \mc{G}^2(r)$, for $r=1, \ldots, n-1$, where $\mc{G}^2(r)$ denotes the indecomposable $r$-dimensional Weil-Deligne representations.  We compute 
 \begin{equation*}
 \begin{split}
 L(s, \varphi_\sigma\otimes \tau)  &=  L(s, \bigoplus_{i=1}^r \varphi_{Q(\Delta_i)} \otimes \tau)  \\
            &= \prod\limits_{i=1}^r L(s, \varphi_{Q(\Delta_i)}\otimes \tau) \\
            & =\prod\limits_{i=1}^r L(s, Q(\Delta_i)\times \pi_\tau)          
 \end{split}
 \end{equation*}
 where in the last equality we have used LLC and written $\pi_\tau$ for the representation of $\GL_r(F)$ corresponding to $\tau$. Meanwhile  $$L(s, \varphi_\pi\otimes \tau) = L(s, \pi\times \pi_\tau)$$ by the LLC, and this is equal to the above by of \cite[\S2.8.]{henniart2002caracterisation}, using the fact that $Q(\Delta_1, \ldots, \Delta_r) = J(Q(\Delta_1), \ldots, Q(\Delta_r))$, in the notation of that section.  \end{proof}

 The above lemma shows how ``essentially square-integrable support'' of a representation of $\GL_n(F)$ determines its Langlands parameter. We need to also understand the relationship between Langlands parameters and supercuspidal supports. The following definition originates in \cite{SBC}.

\begin{definition} Let $G$ be a connected reductive group over a local field $F$. We say LLC+ holds for $G$ if LLC holds for $G$ and all its $F$-Levi subgroups, and that we have a compatibility with parabolic induction as follows. If $M\subset G$ is a Levi subgroup then $^L M \subset {^L} G$, well-defined up to $\widehat{G}$-conjugacy. If $\sigma\in \Pi(M/F)$, and if $\pi\in \Pi(G/F)$ is an irreducible subquotient of the normalized parabolically induced representation $i_P^G(\sigma)$, where $P=MN$ is an $F$-parabolic of $G$ with $F$-Levi subgroup $M$, then $\varphi_\pi|_{W_F}:W_F\ra {^L}G$ and $\varphi_\sigma|_{W_F}:W_F\ra {^L}M \subset {^L}G$ are $\widehat{G}$-conjugate. 

\end{definition}
 
 For $\GL_n(F)$ this is handled in \cite[Theorem 1.2.]{scholze2013local}, so LLC+ for $\GL_n(F)$ is known.

  \section{The LLC and LLC+ for $\GL_r(D)$} \label{section LLCforGLD}
  
  \subsection{The Definition of LLC for $\GL_r(D)$}

The purpose of this section is to address the corresponding two compatibilities of LLC for the inner form $G=\GL_r(D)$, where $D$ is a $d^2$-dimensional $F$-central division algebra. So we will carefully state what the LLC is for $G$, and then discuss the analogous compatibility conditions. Now $^L G = \widehat{G}\rtimes W_F=\GL_n(\C)\times W_F$ is its $L$-group, where $n=rd$. In this case LLC asserts that there is a bijection $$\Pi(G/F)\ra \Phi(G/F)$$ where $\Pi(G/F)$ is the set of isomorphism classes of irreducible smooth representations of $G(F)$, and $\Phi(G/F)$ is the set of $\widehat{G}$-conjugacy classes of admissible homomorphisms $\varphi:W_F'=W_F \ltimes \C \ra {^L}G$, and this bijection should satisfy various desiderata (\cite{Borel}, section 10). The notion of admissibility for the Langlands parameters depends on the particular inner form $G$; we refer to Theorem \ref{characterizationtheorem}. 
 

Let $$JL:\Pi^2(G)\ra \Pi^2(G^*)$$ be the Jacquet-Langlands correspondence; we use the same letter for all $r$ and $n$. It is well-known that $JL$ is characterized by a character identity on corresponding regular conjugacy classes; see for example \cite{bdkv1984repr}. Explicitly, this states that if $\Theta_\pi$ is the Harish-Chandra character of $\pi$, then we have 
$$\Theta_\pi(g) = \Theta_{JL(\pi)}(g^*) (-1)^{n-r}$$ whenever $g\in G$ and $g^*\in G^*$ correspond. See section \ref{sectionmatchingdist} below for this notion. 
We recall the necessary parametrization of irreducible representations of $G$, as described in \cite{badulescu2007jacquet}. Consider the absolute value of the reduced norm 
$$\nu:= |Nrd|_F: G\to \C^\times.$$ We use the same letter for all $r$. Let $\pi$ be a smooth irreducible representation of $G$. As for $\GL_n(F)$, it is the unique irreducible quotient of a ``standard representation'' $i_{P}^{G}\sigma$. This means that $\sigma = \otimes_{i=1}^k\sigma_i$ where $\sigma_i$ are essentially square-integrable with a condition on the $\sigma_i$ 
parallel to the ``does not precede'' condition on segments above. Explicitly, $\sigma_i = \nu^ {e_i} \sigma_i^u$ for some real number $e_i$ and square-integrable representation $\sigma_i^u$, and the condition is that $e_i \geq e_{i+1}$ for all $1\leq i\leq k-1$. Then the standard representation is uniquely determined by $\pi$ up to permutations of the $\sigma_i$ which do not change the $e_i$. See \cite[\S 2.1.]{badulescu2007jacquet}, for example. 
 
\begin{definition}
We define the Langlands parameter of $\pi\in \Pi(\GL_r(D))$ to be $\varphi_\pi:=\bigoplus \varphi_{\text{JL}(\sigma_i)}$. 
\end{definition}

 Lemma \ref{LangQuotLLC} shows this is consistent with the case $D=F$. With this definition we will show the + part of LLC+. From \cite{badulescu2007jacquet} we have an extension of $JL$ to the Grothendieck groups $R(G)$ and $R(G^*)$ of finite length representations as follows. Let $\B_G$ (and similarly $\B_{G^*}$) be the set of standard representations $i_L^G \sigma$. Then $\B_G$ is a base for $R(G)$, and we define 
 $$JL: \B_G\hookrightarrow \B_{G^*}$$  
 $$i_L^G \sigma \mapsto i_{L^*}^{G^*}JL(\sigma)$$ 
 where $L^*$ is a Levi subgroup whose conjugacy class is associated to the conjugacy class of $L$. We remark that this map obviously takes standard representations of $G$ to those of $G^*$, since they have the same form. Concretely, $L\cong \prod_{i=1}^k\GL_{r_i}(D)$ and $L^*\cong \prod_{i=1}^k \GL_{r_i d}(F)$ with $\sum\limits_{i=1}^k r_i = r$. We extend linearly to an injective homomorphism of free abelian groups
 $$JL: R(G)\hookrightarrow R(G^*).$$ 
 However, $JL$ does not map irreducibles to irreducibles, and in fact may send an irreducible representation to a non-genuine virtual representation. Another natural function considered in \cite{badulescu2007jacquet} is 
 $Q:\Pi(G)\hookrightarrow \Pi(G^*),$ defined to make the diagram  $$\xymatrix{ \B_G \ar[r]^{JL} \ar[d]^{\cong} & \B_{G^*} \ar[d]^{\cong}\\ \Pi(G)  \ar[r]^{Q} & \Pi(G^*)  }$$ commute, where the vertical arrows are Langlands quotients. Our definition and the above lemma imply that 
 $\varphi_{\pi}=\varphi_{Q(\pi)}.$ The map $JL$ is also compatible with arbitrary parabolic inductions. See \cite[Theorem 3.6.]{badulescu2007jacquet} and its proof. For future use we also note that there is a natural surjective map 
 $$LJ: R(G^*)\to R(G)$$ 
defined on the base of standard representations to be an inverse to $JL$ on its image, and zero on the complement.

 \begin{Remark}
 It is conjectured in \cite{badulescu2007jacquet} that $LJ$ takes irreducibles to irreducibles up to sign, or to zero. This remains open in general. 
 \end{Remark}

\begin{proposition}
With the definition above, the + part of LLC+ holds for $\GL_r(D)$. 
\end{proposition}

This proposition is used in section 13.2. of \cite{HainesSatake}.

\begin{proof}
Let $\pi$ be the unique irreducible quotient of $i_P^G\sigma$ as above. Since $\pi$ and $\sigma$ have the same supercuspidal support, we may assume $\pi=\sigma$, i.e., that $\pi\in \Pi^2(G)$. Following \cite{badulescu2007jacquet} (or \cite{tadic1990induced}), the supercuspidal support of $\sigma$ has the form $\tau =\rho \otimes \nu^s \rho \otimes\ldots \otimes \nu^{(k-1)s}\rho$, where $\nu=|Nrd|_F$ and $s$ is an integer determined by $\sigma$. Further, $\sigma$ is the unique irreducible quotient of $i_P^G\tau $. We must show that $$\varphi_\sigma|_{W_F} = \varphi_\tau|_{W_F}$$ or equivalently 
$$\varphi_{JL(\sigma)}|_{W_F}=\varphi_{JL(\tau)}|_{W_F}$$ which is then equivalent to showing that $JL(\sigma)(=Q(\sigma) )$ and $JL(\tau)$ have the same supercuspidal support. Towards this end, let $P=LU$ be a Levi decomposition of $P$, and write 
$$i_L^G \tau = \sigma + \sum a_i i_{L_i}^G \tau_i$$ in $R(G)$, where all the summands on the RHS are standard representations, and the $L_i$ are all proper $F$-Levi subgroups of $G$, since $\sigma$ is the only essentially square-integrable subquotient of $i_L^G \tau$. Note that $i_P^G \tau$ is not itself a standard representation if $k>1$. Applying $JL$ to this equation gives $$JL (i_L^G \tau) = JL (\sigma)  + \sum a_i JL( i_{L_i}^G \tau_i)$$ hence 
$$i_{L^*}^{G^*} JL(\tau) = Q(\sigma) + \sum a_i i_{L_i^*}^{G^*} JL(\tau_i)$$ where we have used that $JL$ commutes with parabolic induction for the LHS, and the definition of $JL$ for the RHS. Now $Q(\sigma)$ is itself essentially square-integrable and hence a standard representation, and standard representations form a base. All the Levi subgroups $L_i^*$ are proper, so $Q(\sigma)$ is not a subquotient of any $\iota_{L_i^*}^{G^*} JL(\tau_i)$, and thus $Q(\sigma)$ is indeed a subquotient of $i_{L^*}^{G^*} JL(\tau)$. Hence LLC+ for $\GL_r(D)$ follows from LLC+ for $\GL_n(F)$. 
\end{proof}

Note that $JL:\Pi^2(G)\to \Pi^2(G^*)$ does not preserve supercuspidals, though the inverse of $JL$ does. For ease of reference, we record the following result, which characterizes the image of $JL$ restricted to supercuspidal representations. 

\begin{lemma} \label{BaduLemma}
The representation $\sigma\in \Pi^2(G)$ is supercuspidal iff $JL(\sigma)$ has vanishing Jacquet modules with respect to all Levi subgroups $L^*\subset \GL_n(F)$ whose conjugacy classes correspond to those of some Levi subgroups of $\GL_r(D)$.
\end{lemma}

\begin{proof}
If we identify $L^*$ with a product of $\prod_{i=1}^k\GL_{m_i}(F)$, then the conjugacy class of $L^*$ corresponds to a conjugacy class of Levi subgroups $L$ in $G$ iff $d$ divides each $m_i$. Then we apply \cite[Lemma 2.4.]{badulescu2007jacquet}, which shows that $JL(\sigma)\in \Pi^2(G^*)$ has Jacquet modules $r_{L^*}^{G^*}JL(\sigma) = JL(r_L^G \sigma)$ if $L$ corresponds to $L^*$. This implies the result. 
\end{proof}

\begin{cor}\label{scsupport}
The supercuspidal support of $JL(\sigma)$, for $\sigma\in \Pi(G)$ supercuspidal, is of the form $(L^*, \tau)_G$ where $L^*$ is conjugate to a Levi subgroup of the form $ \GL_{r \frac{d}{k}}(F)^{k}$, where $(r, k)=1$ and $\tau\cong \rho\otimes \nu\rho\otimes \cdots \nu^{k-1}\rho$, for $\rho$ a supercuspidal $\GL_{r \frac{d}{k}}(F)$-representation. 
\end{cor}



\subsection{Compatibility with Twists}\label{section LLCtwist}

The LLC for $\GL_n(F)$ has the following well-known property: if $\chi$ is a character of $F^\times$ and $\pi$ is an irreducible representation of $\GL_n(F)$, then $$\varphi_{\chi\pi}=\varphi_\pi \otimes (\chi\circ Art_F^{-1}).$$ The above definition, the identification of characters of $\GL_r(D)$, $\GL_n(F)$, and $F^\times$, and the compatibility of Jacquet-Langlands and Langlands quotients with twists, makes it clear that the same holds true for the inner forms. Explicitly, if $\pi$ is the Langlands quotient of $\sigma_1\times \cdots \times \sigma_k$ then $\chi\pi$ is the Langlands quotient of $\chi \sigma_1 \times \cdots \times \chi \sigma_k$ and so $\varphi_{\chi \pi} = \oplus \varphi_{\chi JL(\sigma_i)} = \oplus \varphi_{JL(\sigma_i)} \otimes (\chi\circ Art_F^{-1}) = \varphi_\pi \otimes (\chi\circ Art_F^{-1})$. 

A Langlands parameter $\varphi_\pi$ can also be thought of as a 1-cocycle in $H^1(W_F', \hat{G})$, see \cite{SBC}. An unramified character $\chi$ on $G(F)$ can be regarded as an element $$z_\chi\in H^1(W_F/I_F, Z(\hat{G})^{I_F}) = H^1(\langle\Phi\rangle, Z(\hat{G})^{I_F}).$$ Then also $\varphi_\pi z_\chi\in H^1(W_F', \hat{G})$ and should equal $\varphi_{\chi\pi}$. 
The above shows that this holds for $\GL_r(D)$ if it holds for $\GL_n(F)$, which it is known to do.

\subsection{Characterizations of LLC for an inner form} \label{section LLCcharacterization}

The above results are summarized in the following theorem. 


\begin{theorem} \label{characterizationtheorem}
Fix $D$ as above. 

a) There is a a unique family of bijections $$\rec_r:\Pi(\GL_r(D)) \to \Phi(\GL_r(D))\subset \Phi(\GL_n(F))$$ such that: 1. $\rec_r|_{\Pi^2(\GL_r(D))} = \rec^*_n\circ JL_r$ where $\rec_n^*$ is the LLC for $\GL_n(F)$ and $JL_r$ is the  Jacquet-Langlands correspondence; and 2. If $\pi$ is the Langlands quotient of $\sigma_1\times \ldots \times \sigma_k$ then $\rec_r(\pi) = \rec_{r_1}(\sigma_1)\oplus \cdots \oplus \rec_{r_k}(\sigma_k)$. 

b) The image of these maps consist of those parameters which decompose as direct sums of indecomposable Weil-Deligne representations $W'_F \to \GL_m(\C)$ where $d$ divides $m$. 

c) If we postcompose $\rec_r$ with restriction to the Weil group $W_F$, then the resulting family is characterized by compatibility with Jacquet-Langlands and parabolic induction (in the sense of LLC+). The image of $\res|_{W_F}\circ \rec_r$ is those representations $W_F\to \GL_{rd}(\C)$ which factor through a Levi subgroup of the form $\prod\limits_{i=1}^k \GL_{m_i \frac{ d}{  q_i}}(\C)^{q_i}$ for some coprime integers $m_i$ and $q_i$ such that $\sum\limits_{i=1}^k m_i = r$. 
\end{theorem}

\begin{proof}
Part a) is immediate from the definition, the determination of irreducible representations as Langlands quotients, and the well-known unicity of the Local Langlands Correspondence for $G^*$ and the Jacquet-Langlands correspondence. Part b) is immediate from the definition, since $\rec_{r_i}(\sigma_i) = \rec_{d r_i}^* (JL_{r_i}(\sigma_i) )$ is a homomorphism $W_F' \to \GL_{r_i d}(\C)\rtimes W_F$. Part c) follows from the fact that essentially square-integrable representations of $G^*$ have supercuspidal supports of the form $\tau\otimes \nu\tau \otimes \cdots \nu^{\ell-1} \tau$, and Corollary \ref{scsupport} applied to Levi subgroups of $G$. Since every parameter of $G$ is a direct sum of parameters associated to essentially square-integrable representations of various $\GL_m(F)$, and we know LLC+ for $G$ and all its Levi subroups, the result follows. 
\end{proof}

\section{A morphism of Bernstein Centers} \label{section morphismT}
 
 \subsection{Constructing $\mf{T}$}
 
The primary goal of this section is to construct and give some basic properties of a morphism connecting the Bernstein varieties of $G$ and $G^*$; much of the rest of the paper will be devoted to its study. For basic concepts of the Bernstein center, we refer to \cite{SBC}. First we need a small lemma. 

\begin{lemma} \label{normalizer}
  If $M=\prod\limits_{i=1}^k \GL_i(D)^{r_i}\subset G$ then there is a canonical isomorphism $N_G(M)/M\cong \prod\limits_{i=1}^k S_{r_i}$. 
 \end{lemma}
 
 \begin{proof}
 This is a simple exercise in linear algebra, following easily from the equality $N_G(M)= N_G(Z(M))$. 
  \end{proof}

 If $M\cong\prod\limits_{i=1}^k \GL_{r_i}(D)$ is a standard Levi subgroup of $G$, and $\sigma$ is a supercuspidal irreducible $M$-representation, then let $M^*\cong \prod\limits_{i=1}^k \GL_{r_i d}(F)$ be the associated standard Levi subgroup of $G^*$, and let $M_\sigma^* \cong \prod\limits_{i=1}^k \GL_{\frac{r_i d}{  q_i}}(F)^{q_i}$ be the standard Levi subgroup of $M^*$, such that $JL(\sigma)$ has supercuspidal support $(M_\sigma^*, \rho_\sigma)_{G^*}$, for some integers $q_i$. For square-integrable $\sigma$, supercuspidality of $\sigma$ is equivalent to the requirement that $(q_i, r_i)=1$. If $\sigma=\otimes \rho_i$ then $JL(\sigma) = \otimes JL(\rho_i)$, and $JL(\rho_i)\in \Pi^2(\GL_{r_i d}(F))$ has supercuspidal support 
$$\left(\GL_{\frac{r_i d}{  q_i}}(F)^{q_i}, \pi_i\right)_{\GL_{r_i d}(F)}$$ where $\pi_i:=\tau_i\otimes \nu\tau_i \otimes\cdots\otimes \nu^{q_i-1} \tau_i$,
 by page 53 of \cite{tadic1990induced}. So $JL(\sigma)$ has supercuspidal support $(M_\sigma^*, \rho_\sigma)_{G^*}$ where $\rho_\sigma:=\otimes_{i=1}^k \pi_i$
 

\begin{proposition}\label{5.1Prop}
The assignment 
$$\mf{T}: (M, \sigma)_G\mapsto (M_\sigma^*, \rho_\sigma)_{G^*} : \mf{X}_G\ra \mf{X}_{G^*} $$
where $(M^*_\sigma,\rho_\sigma)_{G^*}$ is the supercuspidal support of $JL(\sigma)$, is a well-defined morphism of varieties. Hence it induces a ring homomorphism
$$\mf{Z}(G^*)\ra\mf{Z}(G)$$ 
of Bernstein centers, which we shall also call $\mf{T}$. This homomorphism can be characterized as follows: if $Z^*\in \mf{Z}(G^*)$ then for all Levi subgroups $L\subset G$ and for every supercuspidal $\tau\in \Pi(L)$, the distribution $\mf{T}(Z^*)$ acts on $\iota_L^G(\tau)$ by the same scalar by which $Z^*$ acts on $\iota_{L^*}^{G^*}JL(\tau)$.
\end{proposition}
\begin{proof}
By definition, $(M, \sigma)_G$ is the $G$-conjugacy class of the pair $(M, \sigma)$ where $M$ is a Levi subgroup of $G$ and $\sigma$ is a supercuspidal $M$-representation. To show that the assignment $(M, \sigma)_G\mapsto (M_\sigma^*, \rho_\sigma)_{G^*}$ is well-defined, it suffices to show that if $w\in N_G(M)/M$, then $(M, \sigma)_G$ and $(M, {^w}\sigma)_G$ have the same image. It is enough to show that $JL(\sigma)$ is conjugate to $JL(^w\sigma)$. But in fact $JL(^w \sigma) \cong {^w}JL(\sigma)$, via the identification $N_G(M)/M \cong N_{G^*}(M^*)/M^*$ implied by Lemma \ref{normalizer} and the character identity that characterizes $JL$. Hence the assignment gives a well-defined function. 

Before proving this is a morphism, we note the requisite fact that the map of varieties preserves connected components. This follows from the fact stated in \cite{badulescu2007jacquet} that $JL$ commutes with unramified characters, using the canonical identification $X(G)\cong X(G^*)$ of unramified characters. This fact is easy to show from the characterization of Jacquet-Langlands by the character identity. Explicitly, $\chi JL(\sigma)$ has supercuspidal support $(M_\sigma^*, \chi_{|M_\sigma^*} \rho_\sigma)$. We remark that by p.53 of \cite{tadic1990induced}, the $q_i$ are unchanged by the twist.

We will prove that the restriction of $\mf{T}$ to any connected component is a morphism. Fix $(M,\sigma)$ where $M$ is a standard Levi subgroup in $G$ and $\sigma$ is a supercuspidal $M$-representation. Write 
$M=\prod\limits_{i=1}^k \GL_{r_i}(D)$, $M^*=\prod\limits_{i=1}^k \GL_{r_i d}(F)$, and $M_\sigma^*=\prod\limits_{i=1}^k \GL_{r_i d/q_i}(F)^{q_i}$, as above, so that $M_\sigma^*$ is a Levi subgroup of $M^*$ and $(M_\sigma^*, \rho_\sigma)_{G^*}$ is the supercuspidal support of $JL(\sigma)$. Also write 
$$\mf{s}:=[M,\sigma]_G $$
$$\mf{s}_M:=[M,\sigma]_M $$
$$\mf{s^*}:=[M_\sigma^*, \rho_\sigma]_{G^*} $$
$$\mf{s}^*_{M_\sigma^*} := [M_\sigma^*,\rho_\sigma]_{M_\sigma^*}$$
for the relevant inertial supports. Noting that $M$ is an inner form of $M^*$, we can consider the morphism $X(M)\cong X(M^{*})\xra{\res}X(M_\sigma^*)$, which we also call $\res$ and denote by $\chi\mapsto \chi|_{M_\sigma^*}$. This is in fact a closed immersion: it is a product of diagonal morphisms $(\C^\times)^k\hookrightarrow (\C^\times)^{\sum\limits_{i=1}^k q_i}$. We have the following commutative diagram with surjective vertical arrows
$$\xymatrix{ X(M) \ar[r]^{\res} \ar[d] & X(M_\sigma^*) \ar[d]\\ \mf{X}_{\mf{s}_M}  \ar[r]  \ar[d] &  \mf{X}_{\mf{s}^*_{M_\sigma^*}} \ar[d] \\ \mf{X_{\mf{s}}} \ar[r]^{\mf{T}} & \mf{X_{\mf{s^*}}}   }$$ where the first vertical arrows are $\chi\mapsto (M, \chi\sigma)_M$ and $\chi^*\mapsto (M_\sigma^*, \chi^* \rho_\sigma)_{M_\sigma^*}$, the second ones are the structure maps for the Bernstein components. For example, the first one is $(M, \chi \sigma)_M\mapsto (M, \chi \sigma)_G$; see the ``abstract constant term'' below. The middle horizontal map is $(M, \chi \sigma)_M\mapsto  (M_\sigma^*, \rho_{\chi\sigma})_{M_\sigma^*}= (M_\sigma^*, \chi|_{M_\sigma^*} \rho_\sigma)_{M_\sigma^*}$. These vertical maps are the ones that define variety structures on the Bernstein components. Thus $\mf{T}$ is indeed a morphism. 

The characterization is immediate from the definition, and the identification of the Bernstein center as the categorical center. \end{proof}

\begin{theorem}\label{5.1.Thm}
The morphism $\mf{T}$ is a closed immersion. 
\end{theorem}
 \begin{proof}

 It suffices to consider the case where $M=\GL_m(D)^k$ and $\sigma = \tau\otimes \cdots \otimes \tau$, since the other cases are effectively tensor products of this one. So $M^*_\sigma = \GL_{m \frac{d}{\ell}} (F)^{k\ell}$ for some divisor $\ell$ of $d$, coprime to $m$, and $\rho_\sigma = \otimes_{j=0}^{\ell-1} \otimes_{i=1}^k \nu^j \rho$ for some supercuspidal representation $\rho$. Observe that if we identify $X(M)\cong (\C^\times)^k$, then $$stab_\sigma := \{ \chi\in X(M): \chi\sigma \cong \sigma  \}$$ is isomorphic to $\mu_{m/s}^k$ for some $s$ dividing $m$. This follows from $stab_\sigma\cong stab_{\tau}^k$, and $stab_\tau \subset \mu_m\subset \C^\times$. To see the latter, suppose $\xi \tau \cong \tau$ as $\GL_m(D)$-representations, with $\xi$ unramified, and consider central characters. 
 
 Similarly, if we identify $X(M_\sigma^*)\cong (\C^\times)^{k\ell}$, then we claim that $stab_{\rho_\sigma}$ is isomorphic to $\mu_{m/s}^{k \ell}$, for the same $s$. Since $\rho_\sigma$ is a tensor product of unramified twists of a single representation $\rho$, the group $stab_{\rho_\sigma}\cong stab_\rho^{k\ell}$ is certainly of the form $\mu_{md/\ell t}^{k\ell}$ for some $t$. But if $\chi\rho\cong \rho$ for $\chi\in X(\GL_{md/\ell}(F) )$, then $\chi\otimes \cdots \otimes \chi \in X(M_\sigma^*)$ is the restriction of a character of $M^*= \GL_{md}(F)^k$, hence comes from a character $\xi$ of $M$, and $\xi$ necessarily stabilizes $\sigma$. 
 
 Let $\sigma$ be the basepoint in identifying $\mf{X}_{\mf{s}_M}$ with $X(M)/stab_\sigma$, and let $\rho^{\otimes k\ell}$ play the same role on the quasisplit side. Observe that the choice of basepoints ensures that the action of the symmetric groups $$W[\mf{s}] := \{w\in N_G(M)/M: {^w}\sigma\cong \chi\sigma \text{ for some } \chi\in X(M)  \}\cong S_k$$ and $W[\mf{s^*}] \cong S_{k\ell}$ are the ordinary actions. Passing to $S_k$ and $S_{k\ell}$ invariants, the induced homomorphism of rings of regular functions takes the form 
 $$\C[z_1^{\pm m/s}, \ldots, z_{k\ell}^{\pm m/s }] ^{S_{k\ell}} \to \C[ t_1^{\pm m/s}, \ldots, t_k^{\pm m/s}]^{S_k}$$
  $$ p_i(z_1^{m/s}, \ldots, z^{m/s}_{k\ell})  \mapsto \frac{1-q^{-ik\ell m/s}}{1-q^{-im/s}}p_i(t_1^{m/s}, \ldots, t_k^{m/s}) $$ where the $p_i$ are the power sum symmetric functions.
 As this is clearly surjective, the conclusion follows. 
 \end{proof}

 \subsection{The Cocenter} \label{section cocenter}
  
  To complete the picture, we consider another related perspective. Let $J_G\subset \Hz(G)$ be the subspace spanned by functions of the form $f- {^g} f$, for $g\in G$ and $f\in \Hz(G)$. Let 
  $$\HH(G)=\Hz(G)/J_G$$ 
  be the cocenter, and similarly for $G^*$. By a theorem of Kazdhan (\cite{Kaz}) we have an inclusion of $\C$-vector spaces into the dual of the complexified Grothendieck group of finite length representations
  $$\HH(G^*)\hookrightarrow R(G^*)_\C^\vee$$ 
  $$\bar{f}\mapsto (\pi \mapsto \tr \pi(f))$$
 where $f$ is any lift to $\Hz(G^*)$ of $\bar{f}$, whose image is characterized by the trace Paley-Wiener theorem, described below. By considering the same inclusion for the inner form $G$ and the map 
  $$JL^\vee: R(G^*)^\vee\to R(G)^\vee$$ dual to $JL$ above, we check below that this induces a map 
  $$\HH(G^*)\to \HH(G)$$ which we will compare to the map of Bernstein centers considered above. Here the traces are taken with respect to measures $dg$ and $dg^*$ which are compatible in the sense of \cite{Kott}. So suppose $\phi\in R(G^*)^\vee_\C$ is in the aforementioned image. By the trace Paley-Wiener theorem (\cite[Appendix B]{bdkv1984repr}), this is equivalent to saying that $\phi$ has the following two properties:
  
  1) There is a compact open subgroup $K^*\subset G^*$ such that $\phi(\pi)=0$ if $\pi\in \Pi(G^*)$ with $\pi^{K^*}=0$.  
  
  2) For all Levi subgroups $M^*$ and finite-length $M^*$-representations $\sigma$, the map 
  $$\chi\mapsto \phi(i_{M^*}^{G^*}(\chi \sigma)): X(M^*)\to \C$$ 
  is an algebraic morphism. 
  
  The trace Paley-Wiener theorem applied to $G$ gives an analogous characterization for elements of $R(G)_\C^\vee$ which come from $\HH(G)$. We must show that $$JL^\vee(\phi):\pi\mapsto \phi(JL(\pi))$$ has the same two properties for the inner form. 
  It will be more straightforward to demonstrate an equivalent condition to property 1). We will prove the following lemma for $G$; the lemma for $G^*$ is then a special case.

  \begin{lemma}
  The functional $\phi\in R(G)_\C^\vee$ satisfies 1) iff there exist finitely-many inertial classes $\mf{s}_i=[M_i, \sigma_i]_{G}$ such that $\phi(\pi)=0$ if $\pi$ is an irreducible representation with inertial support $\mf{s}\neq \mf{s_i}$ for all $i$. 
  \end{lemma} 
  
  \begin{proof}
  Suppose that $K\subset G$ is a compact open subgroup such that $\phi(\pi)=0$ for all $\pi\in \Pi(G)$ with $\pi^K=0$. From the theory of the Bernstein center \cite{Roche}, we have that $e_{K}= e_{\mf{s}_1}+\ldots + e_{\mf{s_m}}$ where $e_{\mf{s_i}}$ is the idempotent projecting onto the corresponding component. So if $\phi$ satisfies condition 1) then it satisfies the other condition. 
  
  Conversely, suppose $\phi$ is supported on finitely-many Bernstein components corresponding to $\mf{s}_i$ as above. If $P_i=M_iN_i$ is a parabolic subgroup with Levi factor $M_i$, and $r_{M_i}^G$ denotes Jacquet restriction, then in $R(G)$ we have 
  \begin{equation*}
r_{M_i}^G \iota_{M_i}^G (\chi\sigma_i)  = \sum {^g} (\chi\sigma_i)
  \end{equation*}
  where the sum is over $g\in N_G(M_i)/M_i$.  This formula is a special case of \cite{Roche} Lemma 1.7.1.1. Now choose compact open subgroups $K_{M_i}$ of $M_i$ such that every subquotient of $r^G_{M_i} \iota_{M_i}^G \chi\sigma_i$, in particular $\chi\sigma_i$, has $K_{M_i}$-fixed vectors. The sum formula above shows that all subquotients of $r_{M_i}^G \iota_{M_i}^G \chi \sigma_i$ have $K_{M_i}$-fixed vectors for all unramified characters $\chi$ of $M_i$. 
  We may further choose $K_{M_i}$ small enough so that there exists a compact open subgroup $K_i$ of $G$ with Iwahori factorization $K_i=(N_i\cap K_i)(M_i\cap K_i)(\overline{N_i}\cap K_i)$ and $K_{M_i} = M_i\cap K_i$. Now suppose $\pi$ is an irreducible representation with $\phi(\pi)\neq 0$. Suppose $\mf{s}_i$ is the inertial support of $\pi$, for some $1\leq i\leq m$, so $\pi_{N_i} = r^G_{M_i}(\pi)$ is a subquotient of $r_{M_i}^G \iota_{M_i}^G \chi \sigma_i$ and some unramified character $\chi$. We deduce that $\pi_{N_i}^{K_{M_i}}\neq 0$. By Jacquet's Lemma, the canonical map $\pi^{K_i}\to \pi_{N_i}^{K_{M_i}}$ is onto, so $\pi^{K_i}\neq 0$. Taking $K =\bigcap_{i=1}^m K_i$ yields the desired result. 
  \end{proof}
  

  \begin{proposition}
  The map $JL_\C^\vee : R(G^*)_\C^\vee \to R(G)^\vee_\C$ restricts to give a linear map $\overline{\mf{T}}:\HH(G^*)\to \HH(G)$. This map is characterized by the identity 
  $$\tr ( \overline{\mf{T}}(\bar{f^*}): \pi )= \tr( \bar{f^*} : JL(\pi)) $$ for all irreducible, or equivalently for all finite length, representations $\pi$ of $G$. 
  \end{proposition}
  

     \begin{proof}
      Assume first that $\phi\in R(G^*)^\vee_\C$ is such that  $\phi(\sigma)=0$ unless $\sigma$ is a linear combination of irreducible representations with inertial supports in a fixed set $\mf{s}^*_i$, for $i=1, \ldots, m$. We claim that for all irreducible representations $\pi$, $\phi(JL(\pi)) =0$ unless $\pi$ has inertial support in some finite set $\{\mf{s}_j\}$. But we can take this to be the set of preimages of the $\mf{s}_i^*$ induced by the morphism $\mf{T}$, which we note may be empty. Here we are using the compatibility of $JL$ with arbitrary parabolic inductions. See \cite[Theorem 3.6.]{badulescu2007jacquet} for the proof of this. 
       
       Now assume that $\phi\in R(G^*)^\vee_\C$ satisfies condition 2) above. Let $M\subset G$ be a Levi subgroup and let $\sigma$ be a finite-length $M$-representation. We claim 
       $$\chi\mapsto \phi(JL(i_M^G \chi \sigma)):X(M)\to \C$$ 
       is algebraic. But this equals $\phi(i_{M^*}^{G^*} JL(\chi \sigma)) = \phi(i_{M^*} ^{G^*}\chi JL(\sigma))$ using $X(M)\cong X(M^*)$. So this is true by applying property 2) above to $M^*$ and $JL(\sigma)$.
     \end{proof}
  
  \begin{proposition} \label{CocenterSurjective}
  The map $\overline{\mf{T}}$ of cocenters has a section, hence is surjective. 
   \end{proposition}

  \begin{proof}
  Let $LJ^\vee: R(G)_\C^\vee\to R(G^*)_\C^\vee$ be the vector space homomorphism dual to $LJ$. Let $\bar{f}\in \HH(G)\subset R(G)_\C^\vee$ and 
  define $\bar{f^\vee}:=LJ^\vee(\bar{f})\in R(G^*)_\C^\vee$ to be the functional $\pi\mapsto \bar{f}(LJ(\pi)).$ Then since $LJ\circ JL = \id_{R(G)}$, we see that $JL^\vee$ sends $\bar{f}^\vee$ to $\bar{f}$. We will show that $LJ^\vee$ preserves the cocenters. Once this is established, the result follows. 
  
  Let $\phi\in R(G)_\C^\vee$, and suppose that $\phi(\sigma)=0$ unless $\sigma$ is a linear combination of irreducible representations with inertial supports in a fixed set $\{\mf{s}_i\}_{i=1}^m$. This property of $\sigma$ is equivalent to $\sigma$ being a linear combination of standard representations with (all irreducible subquotients having) inertial supports in the same set. We must show $\phi(LJ(\pi))=0$, for $\pi$ an irreducible representation, unless $\pi$ has inertial support in some fixed finite set $\{\mf{s}_j^*\}$. This is equivalent to showing $\phi(LJ(i_{L^*}^{G^*} \tau))=0$ for all standard representations unless $i_{L^*}^{G^*}\tau$ has (all irreducible subquotients having) inertial support in the same finite set. Take this set to be the image of $\{\mf{s}_i\}_{i=1}^m$ induced by the morphism $\mf{T}$. The conclusion follows from the fact that $LJ$ preserves supercuspidal representations and commutes with parabolic induction, both proved in \cite{badulescu2007jacquet}, and that $JL\circ LJ|_{\im(JL)} = \id_{\im (JL)}$.
   
   Next, suppose $M^*\subset G^*$ is a Levi subgroup and $\sigma$ is a finite-length $M^*$-representation. We want to prove that 
   $$\chi\mapsto \phi(LJ(i_{M^*}^{G^*} \chi \sigma)):X(M^*)\to \C$$ is algebraic. If $M^*$ does not transfer to $G$, then this is identically 0, hence regular. If $M^*$ does transfer, this equals $\phi(i_{M}^{G} LJ(\chi \sigma)) = \phi(i_{M} ^{G}\chi LJ(\sigma))$ using $X(M)\cong X(M^*)$. So the function is regular by property 2) above for $M^*$ and $JL(\sigma)$. 
   \end{proof} 

\section{Further Properties of $\mf{T}$} \label{sectionPropertiesofT}

\subsection{Compatibility between $\mf{T}$, constant term homomorphisms, and normalized transfer}

In this section we further study the map $\mf{T}: \mf{Z}(G^*)\to \mf{Z}(G)$. We will make reference to the isomorphisms $$\varprojlim\limits_C \mc{Z}(G, C) \cong \mf{Z}(G)\cong \C[\mf{X}_G]$$ giving realizations of the Bernstein center as an inverse limit of finite-level Hecke algebra centers, as a categorical center, and as a ring of regular functions on the Bernstein variety. We will not make use of its realization as invariant essentially compact distributions until the section on orbital integrals below. See \cite{SBC} for a summary of these realizations. 

Consider the diagram Figure 1. We show in Lemma \ref{diagram} that each square is commutative. 
Here $J$ and $J^*$ are arbitrary parahorics in their respective groups, $M$ denotes a Levi subgroup of $G$, $M^*$ a corresponding Levi in $G^*$, and $J_M=M\cap J$, etc. 

\begin{figure}

\begin{tikzpicture}
  \matrix (m) [matrix of math nodes, row sep=3em,
    column sep=3em]{
   & \C[\mf{X}_{G^*}]& & \C[\mf{X}_{G}] \\
       \C[\mf{X}_{M^*}] & & \C[\mf{X}_{M}] & \\
       & \mc{Z}(G^*, J^*) & & \mc{Z}(G, J) \\
      \mc{Z}(M^*, J^*_{M^*}) & & \mc{Z}(M, J_{M}) & \\};
  \path[-stealth]
    (m-1-2) edge node [above] {$\mf{T}$}(m-1-4) edge node [above] {$C_{M^*}^{G^*}$} (m-2-1)
            edge [ dotted] (m-3-2)
    (m-1-4) edge  (m-3-4) edge node [above] {$C_{M}^{G}$}(m-2-3)
    (m-2-1) edge [-,line width=6pt,draw=white] node [above] {\:\:\:\:\:\:\:$\mf{T}$} (m-2-3)
            edge  (m-2-3) edge  (m-4-1)
    (m-3-2) edge [dotted] node [above] { $\tilde{t}$ \:  \: \:}  (m-3-4)
            edge [ dotted]  node [above] {$c_{M^*}^{G^*}$} (m-4-1)
    (m-4-1) edge node [above] {$\tilde{t}$} (m-4-3)
    (m-3-4) edge node[above] {$c_{M}^{G}$} (m-4-3)
    (m-2-3) edge [-,line width=6pt,draw=white]  (m-4-3)
            edge  (m-4-3);
\end{tikzpicture}
\caption{} \label{fig:M1}
\end{figure}

The normalized transfer homomorphisms $\tilde{t}$ are defined in general in \cite{SBC}; we give a concrete description in our context. From the inclusion $T^*\hookrightarrow M^*$ we get the induced map $$A:\Z^n\cong T^*/T_1^* \ra M^*/M_1^* \xra{} M/M_1\cong \Z^r$$ where $T=\GL_1(F)^n$, $M^*=\GL_d(F)^r$, and $M=\GL_1(D)^r$ is an inner form of $M^*$. The subscript 1's here indicate kernels of the Kottwitz homomorphism. See \cite{SBC} for details. Now via the Bernstein isomorphism we can realize the normalized transfer map as the map of group rings $$\C[T/T_1]^{S_n}\ra \C[M/M_1]^{S_r}$$ $$\sum\limits_{t} a_t t \mapsto \sum\limits_{m} \left(\sum\limits_{t\mapsto m} a_t\delta^{-1/2}_{B^*}(t) \delta_P^{1/2}(m)\right)m $$ where $t\mapsto m$ means that $A(tT^*_1)=mM_1$. Here $P$ is the upper-triangular parabolic with $M$ as Levi factor, and $B^*$ is the upper-triangular Borel containing $T^*$. We reinterpret this as a map $\C[\Z^n]^{S_n}\ra \C[\Z^r]^{S_r}$. Without the normalization, the map above is $(a_1, \ldots, a_n) = (a_{11}, \ldots, a_{1d}, a_{21}, \ldots, a_{rd}) \mapsto (b_1, \ldots, b_r)$ where $b_k = \sum\limits_{i=1}^d a_{ki} = \sum\limits_{i=(k-1)d+1}^{kd} a_i$. 
In our case, if $t$ corresponds to $(a_1, \ldots, a_n)$ and $m$ corresponds to $(b_1, \ldots, b_r)$ then we have 
$$\delta_P^{1/2}(m) = q^{-\sum\limits_{i=1}^r (r+1-2i)db_i/2}$$ $$\delta_{B^*}^{-1/2}(t) = q^{\sum \limits_{i=1}^n (n+1-2i)a_i/2}.$$ The exponent of the product of these together is $\sum\limits_{i=1}^n (n+1-2i)a_i/2 - \sum\limits_{i=1}^r (r+1-2i)db_i/2$ which, after using $dr=n$ and $\sum\limits_{i=1}^r b_i = \sum\limits_{i=1}^n a_i$, is equal to $\frac{1-d}{2}\sum\limits_{i=1}^n a_i - \sum\limits_{i=1}^n ia_i + d\sum\limits_{i=1}^r ib_i$. This, in turn, can be realized as the inner product of $(a_1, \ldots, a_n)$ with the vector $\vec{x}:=(\vec{c}, \ldots, \vec{c})$ (repeated $r$ times) where $\vec{c}:=(\frac{d-1}{2}, \frac{d-3}{2}, \ldots, \frac{1-d}{2})$. In summary, with the above notation, the normalized transfer map is, after restriction to Weyl group invariants, $$\vec{a}:=(a_1, \ldots, a_n) \mapsto q^{\vec{a}\cdot \vec{x}} (b_1, \ldots, b_r).$$ Or if we write this as polynomials, it maps a monomial $\prod\limits_{i=1}^n z_i^{a_i}$ to $q^{\vec{a}\cdot \vec{x}} \cdot \prod\limits_{j=1}^r t_j^{b_j}$. 

\begin{Remark}
The normalized transfer map is a completely canonical map defined between the centers of the parahoric Hecke algebras. In our above description, we have chosen a particular maximal torus, etc., solely to make computations explicit. Also, to write it out completely would involve symmetrizing the above map; we give such a description below. 
\end{Remark} 

The constant term homomorphism $c_M^G$ is defined by $$c_M^G(f)(m) = \delta_P^{1/2}(m)\int_N f(mn)dn$$ for any parabolic $P=MN$ with Levi subgroup $M$, where $dn(J\cap N)=1$. See \cite[\S 11.11.]{SBC} for the proof that $c_M^G$ preserves centers. We remark that $c_M^G$ is always injective since it corresponds to the natural inclusion under the Bernstein isomorphism. The crucial fact about the constant term homomorphism is the equality $$\tr (c_M^G(f) | \sigma )/\dim \sigma^{J_M} = \tr (f | i_P^G \sigma)/ \dim (i_P^G\sigma)^J$$ i.e. both act by the same scalars on their respective spaces (after taking invariants). The ``abstract constant term'' map $C_M^G$ is induced from the morphism $$(L, \sigma)_M\mapsto (L,\sigma)_G: \mf{X}_M\ra \mf{X}_G$$ and $C_M^G$ is generally neither 1-1 nor onto. The morphism inducing $C_M^G$ restricts to a map $\mf{X}_1^M \ra \mf{X}_1^G$, where $\mf{X}^M_1$ corresponds to the inertial class $[\GL_1(D)^r, 1]_M$, and similarly for $\mf{X}^G_1$. The definitions of these maps for the quasisplit group $G^*$ are similar. 

The vertical maps are restrictions to the Iwahori component of the Bernstein variety, which we will write as $\mf{X}_1^G$, in agreement with the notation of \cite{SBC}, via the isomorphism $\mc{Z}(G, J)\cong \C[\mf{X}_1^G]$. Alternatively, they can be viewed as the natural projection maps from the Bernstein center, realized as an inverse limit, to the center of the parahoric Hecke algebra.

\begin{lemma}\label{diagram}
The diagram in Figure 1 commutes. In particular, the homomorphism $\mf{T}$ restricts to the normalized transfer $\widetilde{t}$ of \cite{SBC}. 
\end{lemma}

\begin{proof}

First note that the top-level maps preserve Iwahori blocks - certainly a necessary condition for the diagram to commute. We have already stated this for the abstract constant term. It it also true for $\mf{T}$, since $JL$ takes an unramified character of a minimal Levi in $G$ to a product of twists of Steinberg representations, whose supercuspidal support lies in the Iwahori block of $G^*$. 


\subsubsection{Bottom Square} The bottom square is proved to commute in \cite{SBC}. 

\subsubsection{Top Square} The top square commutes because the corresponding square of varieties commutes. Explicitly, the maps of varieties are $(L, \sigma)_M \mapsto (L, \sigma )_G \mapsto (L_\sigma^*, \rho_\sigma)_{G^*}$ and $(L,\sigma)_M \mapsto (L_\sigma^*, \rho_\sigma)_{M^*} \mapsto (L_\sigma^*,\rho_\sigma)_{G^*}$, hence are equal.

\subsubsection{Constant term square}

Clearly it suffices to prove the commutativity for $G$, since that of $G^*$ is analogous. We cite the fact that if $f\in \mc{Z}(G, J)$ and $\sigma\in \Pi(M)$ has $J_M=M\cap J$-fixed vectors, then $$\frac{\tr(c_M^G(f)| \sigma^{J_M})}{\dim \sigma^{J_M}} = \frac{\tr (f | (i_P^G \sigma)^J)}{ \dim (i_P^G\sigma)^J}.$$ That is, $f$ acts on $(i_P^G \sigma)^J$ by the same scalar by which $c_M^G(f)$ acts on $\sigma^{J_M}$. That makes the commutativity of that square clear, if we regard the Bernstein center as the categorical center. 

\subsubsection{Transfer square} 

Let $(L, \chi)_G\in \mf{X}_1^G$, so that $L\cong (D^\times)^r$ is a minimal Levi subgroup and $\chi=\chi_1\boxtimes\cdots\boxtimes \chi_r$ is an unramified character of $L$. We abuse notation to think of $\chi_i$ as a character of $F^\times$ composed with the reduced norm, so that we may write $$JL(\chi) = \boxtimes JL(\chi_i)  = \boxtimes(\chi_i St_{\GL_d(F)})$$ as a representation of $\GL_d(F)^r$.  Now the Steinberg representation $St_{\GL_d(F)}$ has supercuspidal support $(\GL_1(F)^d, \delta_B^{-1/2})_{\GL_d(F)}$. Thus $\chi_i St_{\GL_d(F)}$ has supercuspidal support $$(\GL_1(F)^d, \chi_i^{\boxtimes d}\cdot \delta_{B_d}^{-1/2})_{\GL_d(F)}.$$ So $JL(\chi)$ has supercuspidal support $(\GL_1(F)^{rd}, \boxtimes_{i=1}^r(\chi_i^{\boxtimes d} \cdot \delta^{-1/2}_{B_d}))_{G^*}$ where we write $B_d$ for the upper-triangular Borel in $\GL_d(F)$. This induces a map $$\C [\mf{X}_1^{G^*}] \cong \C[\mf{X}_1^{\GL_1(F)^n}]^{S_n}\ra\C [\mf{X}_1^G] \cong \C[\mf{X}_1^L]^{S_r}$$ where $\mf{X}_1^L\cong \Hom(L/L_1, \C^\times)$ is a complex torus. This is the normalized transfer: the value of $\boxtimes_{i=1}^r(\chi_i^{\boxtimes d} \cdot \delta^{-1/2}_{B_d})$ on $\diag(\varpi^{a_{11} }, \ldots \varpi^{a_{1d}}, \varpi^{a_{21}}, \ldots, \varpi^{a_{rd}} )$ is equal to $\prod\limits_{i=1}^r\chi_i(\varpi^{a_{i1}+\ldots + a_{id} }) \cdot q^{\sum\limits_{j=1}^d \frac{(d+1-2j)a_{ij}}{2} } = q^{\vec{a}\cdot \vec{x} } \prod\limits_{i=1}^r \chi_i\left(\varpi^{\sum\limits_{j=1}^d a_{ij}}\right) = q^{\vec{a}\cdot \vec{x} } \prod\limits_{i=1}^r \chi_i(\varpi)^{b_i}$ where $\vec{x}$ and $b_i$ are as above. The proof that the front panel commutes is similar, mutatis mutandis, the main difference being the use of Young subgroups instead of just symmetric groups. \end{proof}


\begin{Remark}
The non-normalized transfer described in \cite{SBC} is what we would get by using unnormalized parabolic induction. 
\end{Remark}

 \subsection{Compatibility of $\mf{T}$ and certain distributions} \label{sectionapplication}
 
 We have for each finite dimensional, algebraic representation $(r, V)$ of $^L G={^L} G^*$ a distribution $Z_V$ in the Bernstein center of $G$ defined by $$Z_V(\pi)=\tr(\varphi_\pi (\Phi), V^{I_F}).$$ That it lies in the Bernstein center is a consequence of LLC+. We suppose $V$ is irreducible, so that it is parametrized by a highest weight $\mu$. Really, this is just giving us a representation of the dual group, not the whole $L$-group, but there exists a way to extend to a representation of the $L$-group, described in Lemma 2.1.2 of \cite{kottwitz1984shimura}. We have a function $Z_V^*$ defined analogously. 
 
\begin{cor} \label{Z_V}
 The equality $\mf{T}(Z_V^*)=Z_V$ holds in the Bernstein center. 
\end{cor}

 \begin{proof}
In fact this is a simple consequence of LLC+. To see this, let $(M, \sigma)_G\mapsto(M_\sigma^*, \rho_\sigma)_{G^*}$ where $\rho_\sigma$ is the supercuspidal support of $JL(\sigma)$. By LLC+ for $\GL_n(F)$, we know $\varphi_{\rho_\sigma}(\Phi)=\varphi_{JL(\sigma)} (\Phi)$. By the definition of LLC for $\GL_r(D)$ we know $\varphi_\sigma(\Phi)=\varphi_{JL(\sigma)} (\Phi)$. By LLC+ for $\GL_r(D)$, $Z_V^*$ is the trace of $\varphi_\sigma(\Phi)$ on $V^{I_F}$ and the conclusion follows. 
 \end{proof} 
 
 This gives a generalization of Proposition 7.3.2. in \cite{SBC}, which does the case $n=2$, $r=1$, and $\pi^I\neq 0$, to arbitrary $n$, $r$, and $\pi$.
 
\subsection{The Stable Bernstein Center} \label{sectionSBC}

Let $\mf{X}_G$ and $\mf{Y}_G$ be the Bernstein variety and stable Bernstein variety associated to $G$, and similarly for $G^*$. Let $\mf{Z}(G)$ and $\mf{Z}^{st}(G)$ be the corresponding rings of regular functions, the Bernstein center and stable Bernstein center, respectively. In the case of $\GL_n(F)$, it is shown in \cite{SBC} that the stable and ordinary Bernstein centers coincide, because their corresponding varieties are isomorphic. However the definition of the variety $\mf{Y}_G$ of infinitesimal characters (see $\S5$ of \cite{SBC}) shows that in fact the stable Bernstein centers of $G$ and $G^*$ are isomorphic, since $\mf{Y}_G=\mf{Y}_{G^*}$. On the other hand, the natural map
$$(M,\sigma)_G\to(\varphi_\sigma|_{W_F})_{\hat{G}}:\mf{X}_G\to \mf{Y}_G$$ is in this case neither injective nor surjective. Similarly the corresponding map of rings of regular functions (stable Bernstein centers) is neither surjective nor injective. In summary we have the following two commutative diagrams
$$\xymatrix{ \mf{X}_G \ar[r] \ar[d] & \mf{Y}_{G} \ar[d]^{=}\\ \mf{X}_{G^*}  \ar[r]^{\cong} & \mf{Y}_{G^*}  }$$ 
$$\xymatrix{ \mf{Z}(G)  & \mf{Z}^{st}(G) \ar[l]\\ \mf{Z}(G^*)  \ar[u] & \mf{Z}^{st}(G^*) \ar[u]^{=} \ar[l]^{\cong}  }.$$  
The same diagrams also exist for the conjugacy class of a given Levi subgroup $M\subset G$ and its corresponding conjugacy class of Levi subgroup $M^*\subset G^*$, regarding $M^*$ as an inner form of $M$. We have already said that these are compatible, on the side of the usual Bernstein center, via the abstract constant term maps. They are also compatible with the stable Bernstein centers, via the analogous ``constant term'' map $\mf{Y}_M\to \mf{Y}_G$, and similarly for $G^*$, induced by taking a parameter $\lambda: W_F\to {^L}M$ to $\lambda:W_F\to {^L}M \subset {^L}G$. 

\subsection{The Geometric Bernstein Center} \label{sectionGBC}

Inspired by Corollary \ref{Z_V}, we study the relationship between the geometric Bernstein centers of $\GL_n(F)$ and its inner forms. Recall that this is the algebra generated by the $Z_V$, respectively, $Z_V^*$, inside the stable Bernstein centers of $G$ and $G^*$. There is also the ``$\Phi$-version'' generated by functions $Z_V^\Phi$ defined by $Z_V^{\Phi}(\pi)=\tr(\varphi_\pi(\Phi), V)$, and similarly $Z_V^{*, \Phi}$, which requires choosing a particular Frobenius $\Phi$. We denote these algebras by $\mf{Z}^{geom}(G)$ and $\mf{Z}_{\Phi}^{geom}(G)$, and similarly for $G^*$. Corollary \ref{Z_V} shows that the homomorphism $\mf{T}$ of Bernstein centers preserves both of these geometric Bernstein centers, and at the level of the stable Bernstein centers they are identified. If we transport them to the usual Bernstein centers, it is easy to see that one surjects onto the other. The map can be concretely regarded as restriction of $Z_V$, or $Z_V^{\Phi}$, to the image of $Q:\Pi(G)\hookrightarrow \Pi(G^*)$. 

We  address injectivity of $\mf{T}:\mf{Z}_\Phi^{geom}(G^*)\to \mf{Z}_\Phi^{geom}(G)$. Since $Z_{V\otimes W}^\Phi = Z_V^\Phi \cdot Z_W^\Phi$, the algebra generated by these functions is the same as the vector space they span. Consider a finite sum $\sum_i a_i Z^{*\Phi}_{V_i}$ where the $V_i$'s are irreducible representations. Applying $\mf{T}$ gives $\sum_i a_i Z^{\Phi}_{V_i}$. Convolving with $e_I$ gives an element $e_I * \sum_i a_i Z^{\Phi}_{V_i} \in Z(G, I)$. Then applying the Bernstein isomorphism (see $\S$6 in \cite{SBC}) $Z(G, I)\cong K_0(Rep(\hat{G}))$ gives $\sum a_i V_i$. But the $V_i$ are a basis. So injectivity of  $\mf{T}:\mf{Z}_\Phi^{geom}(G^*)\to \mf{Z}_\Phi^{geom}(G)$ follows.
 
This argument fails for $\mf{Z}^{geom}(G^*)\to \mf{Z}^{geom}(G)$ since it no longer suffices to consider elements of the form $\sum a_i Z_{V_i}$. 

\section{Combinatorial asides}\label{sectioncombasides}

This section will not be used elsewhere in the paper. Corollary \ref{Z_V} above can be given a concrete interpretation, when applied to the Iwahori block. If we rewrite $\C[\Z^n]^{S_n} = \C[t_1^{\pm1}, \ldots, t_n^{\pm 1}]^{S_n}$, then the normalized transfer map can be described as the algebra homomorphism which maps the elementary symmetric functions $$e_k(t_1, \ldots, t_n)\mapsto \sum\limits_{\alpha=(\alpha_1\geq \ldots \geq \alpha_k\geq 0, 0,\ldots)\in \Pi_k} \left[ \prod\limits_{i=1}^k \binom{d}{\alpha_i}_q \right] q^{\sum\limits_{i=1}^k (\alpha_i^2-d) / 2} m_\alpha(t_1, \ldots, t_r)$$ where we sum over the partitions $\Pi_k$ of $k$, the coefficients are the $q$-binomial coefficients 
  and $m_\alpha$ is the symmetric monomial function. 
   In particular, writing it this way makes it clear that symmetric invariants are indeed sent to symmetric invariants. It is less obvious that this is a surjective homomorphism, but this is also true as a consequence of the surjectivity of $\mf{T}$. This can also be written more simply in terms of power sum symmetric functions as $$p_k(t_1, \ldots, t_n) \mapsto \left(\frac{q^{dk/2} - q^{-dk/2} }{q^{k/2}-q^{-k/2} } \right)p_k(t_1, \ldots, t_r) $$ and from this perspective it is clearly surjective.



In general, still on the Iwahori component, if $\mu=(a_1\geq \ldots \geq a_n)$ is any cocharacter, then we have shown above, as a consequence of LLC+, that $\tilde{t}$ takes $z^*_{\mu}=\sum\limits_{\lambda\in S_n\cdot \mu} \lambda$ to 
$$z_{\mu} =\tr(\diag (\eta_1 q^{(d-1)/2}, \ldots, \eta_1 q^{(1-d)/2}, \ldots, \eta_r q^{(1-d)/2}), V_{\mu})$$ where we think of this as a (rational) function of the $\eta_i$. This then is the Schur polynomial 
$$S_\mu(q^{(d-1)/2} t_1, q^{(d-3)/2} t_1, \ldots, q^{(1-d)/2} t_1, q^{(d-1)/2} t_2, \ldots, q^{(1-d)/2} t_2, \ldots, q^{(1-d)/2} t_r).$$ where we write $t_i$ for $\eta_i$. 
We are getting a relationship between this ``renormalized'' $S_\mu$ and $S_\mu(t_1, \ldots, t_n)$. Explicitly, if we write $S_\mu(t_1, \ldots, t_n)=\sum\limits_T  t_1^{a_1}\cdots t_n^{a_n}$, where the summation is over all semistandard Young tableaux $T$ of shape $\lambda$ and $a_i$ counts the occurrences of the number $i$ in $T$, then $$S_\mu(q^{(d-1)/2} t_1, \ldots, q^{(1-d)/2} t_r)= \sum\limits_T q^{\vec{a}\cdot \vec{x}} \cdot \prod\limits_{j=1}^r t_j^{b_j}.$$ For example, $\mu=(k, 0, \ldots, 0)$ corresponds to a symmetric power of the standard representation and we get a (rather simpler) explicit formula. 



\section{Matching Distributions and Functions} \label{sectionmatching}

\subsection{Matching Distributions} \label{sectionmatchingdist}

In this section we will use Proposition \ref{Z_V} to give a means of generating new matching pairs of functions from given ones; see below for the definition of matching pair. {\it We assume for this section that $F$ has characteristic zero}. We constructed above a map of cocenters 
$$\overline{\mf{T}}: \HH(G^*)\to \HH(G)$$ characterized by $$\tr (\pi (\overline{\mf{T}}(\bar{f^*} ) ) )= \tr( JL(\pi) (\bar{f^*})) $$ for all irreducible representations $\pi$ of $G$. Note here we are using $JL$ and not $Q$. We also have a map of geometric Bernstein centers that sends $Z_V^*$ to $Z_V$, essentially by restricting to the image of the map $\mf{T}$, if viewed as regular functions on Bernstein varieties. Since we can also view $Z_V$ and $Z_V^*$ as essentially compact invariant distributions, we can convolve them with elements of the Hecke algebra (see \cite{SBC}). 


We recall some definitions. We say $\gamma^*\in G^*$ is 1) semisimple if its minimal polynomial is separable; 2) regular semisimple (which we will just call regular) if its characteristic polynomial is separable; 3) elliptic if its minimal polynomial is irreducible (and separable). 
For $G$ similar definitions exist in terms of the reduced characteristic and minimal polynomials. For each semisimple $\gamma\in G$ there is a corresponding conjugacy class in $G^*$. The image of this assignment restricted to regular elements in $G$ is those conjugacy classes in $G^*$ whose characteristic polynomials have all irreducible factors of degree divisible by $d$.. We write $\gamma\leftrightarrow \gamma^*$ for this correspondence of conjugacy classes. 

Now we require a comment on measures chosen on both $G$ and $G^*$, in order to properly define the orbital integrals. For the centralizers of $\gamma$ and $\gamma^*$, we have a natural notion of transference of measures, because if $\gamma\leftrightarrow \gamma^*$ then $G^*_{\gamma^*}$ is an inner form of $G_\gamma$. Explicitly, the measures are obtained by transporting invariant differentials of top degree from $G_{\gamma^*}^*$ to $G_\gamma$.  However, the measures chosen on the whole groups $G$ and $G^*$ are {\it not} related in this manner. Instead we choose $dg$ and $dg^*$ to be such that $dg(I) = 1 = dg^*(\GL_n(\Oh_F))$, as is done in \cite{KV}, where $I$ is an Iwahori subgroup of $G$. This is not the same thing as the notion of transference; for example in \cite{laumon} it is shown that if $dg$ and $dg^*$ are transfers, and $r=1$, then $\frac{dg^*(\GL_n(\Oh_F))}{dg(\Oh_D^\times)} = (q-1)(q^2-1)\cdots (q^{n-1}-1) \neq 1$.


\begin{definition} \label{matchingdef}
We say that $f^*\in \Hz(G^*)$ and $f\in \Hz(G)$ have matching orbital integrals (or are associated) if, for all semisimple $\gamma^*\in G^*$, we have $$O_{\gamma^*}(f^*) = \begin{cases}
e(G_\gamma) O_{\gamma}(f) \text{ if } \gamma\leftrightarrow \gamma^* \\
0 \text{                if there is no such } \gamma\in G 
\end{cases}$$ where $e(G_\gamma) = (-1)^{\rank_F (G^*_{\gamma^*}) - \rank_F(G_{\gamma}) }$. 
\end{definition}

If $\gamma\leftrightarrow \gamma^*$ then $G^*_{\gamma^*}$ is the quasisplit inner form of $G_\gamma$. See \cite{kottwitz1983sign} for background on the sign factor $e(G_\gamma)$. Note that if the elements that correspond are regular semisimple, then their centralizers are tori, identified with $F[\gamma]^\times$, and so no sign arises; this fact shall appear in the proof of Theorem \ref{MatchDist} below. However, signs do arise even for $\GL_2(F)$, by considering central elements. In general if considering $\GL_1(D)$ as an inner form of $\GL_n(F)$, then the relevant conjugacy classes of $G^*$ are precisely the elliptic ones. In the regular elliptic case no signs arise, and in the central case the sign is $(-1)^{n-1}$. We remark that we could have put a factor of $e(G^*_{\gamma^*})$ on the left-hand-side of the definition's equality above for symmetry, but since $G^*_{\gamma^*}$ is quasisplit (over $F$) for semisimple $\gamma^*\in \GL_n(F)$ this factor is always 1; $G^*_{\gamma^*}$ can be identified with a product of $\GL_k$'s over finite extensions of $F$. See, e.g., \cite[p.4]{AC}.


\begin{lemma} \label{WIF lemma}  If $f\in \Hz(G)$ and $f^*\in \Hz(G^*)$ are associated, then $\tr(f: \pi)= (-1)^{n-r} \tr(f^*: JL(\pi))$ for all tempered irreducible representations $\pi$ of $G=\GL_r(D)$. Further, $\tr(f^*: \pi^*)=0$ for all tempered irreducible $\pi^*\in R(G^*)$ that are not in the image of $JL: R(G)\to R(G^*)$.   
\begin{proof} Note that tempered irreducible representations of $G$ and $G^*$ can be written as parabolic inductions of essentially square-integrable representations; see for example \cite{tadic1990induced}. So if $\pi\in \Pi(G)$ is tempered, then $JL(\pi)$ is a tempered, irreducible representation as well, since $JL$ commutes with parabolic induction (\cite[Theorem 3.6.]{badulescu2007jacquet}). Furthermore, every tempered representation of $G^*$ so arises.
 By the Weyl Integration Formula,
 \begin{equation*}
 \begin{split}
 \tr (f^*:JL(\pi), dg)  &=  \int_{G^*} f^*(g) \Theta_{JL(\pi)}(g)dg \\
            &=\sum_{T/\sim} |W_T|^{-1} \int_{T} D_{G^*}(t)^2 \Theta_{JL(\pi)}(t) O_t^{G^*}(f^*, dg/dt) dt\\   
            & =  \sum_{T/\sim} |W_T|^{-1} \int_{T^{rs}} D_{G^*}(t)^2 \Theta_{JL(\pi)}(t) O_t^{G^*}(f^*, dg/dt) dt\\    
 \end{split}
 \end{equation*} 
  where we sum over conjugacy classes of maximal $F$-tori of $G^*$, and $T^{rs}\subset T$ denotes the regular semisimple subset. 
Observe that $O_t^{G^*}(f^*)=0$ for $t\in T^{rs}$, by assumption, unless $t$ (or rather its $G^*$ conjugacy class) comes from $G$, and that for such $t$ the centralizing torus $T$ also comes from $G$, so that it makes sense, and is true, that the above equals  
$$\sum_{T/\sim} |W_T|^{-1} \int_{T^{rs}} D_{G}(t)^2 (-1)^{n-r}\Theta_{\pi}(t) O_t^{G}(f, dg'/dt) dt$$ where now the sum and integral are over tori of $G$, and $dg'$ is the compatible measure as described above. Applying the Weyl Integration Formula again, one gets that this is equal to $(-1)^{n-r} \tr(f: \pi, dg')$. 

For the second statement, suppose that $\pi^*\in \Pi(G^*)$ is tempered. Then $\pi^*=i_{P^*}^{G^*}(\sigma)$ for some square-integrable $\sigma\in \Pi^2(M^*)$, where $M^*$ is a Levi factor of $P^*$. The assumption that $\pi^*$ is not in the image of $JL$ is equivalent to the conjugacy class of $M^*$ not coming from the inner form $G$. Because $f^*$, by assumption, is associated to some function on $G$, it follows that the orbital integrals of $f^*$ vanish at all (semisimple) conjugacy classes of $G^*$ which do not come from $G$. The vanishing statement is then (half of) the content of \cite[Lemma 3.3.]{apoitiers2003resultat}. 
\end{proof}
\end{lemma}

\begin{theorem} \label{MatchDist}
Suppose $f$ and $f^*$ are associated, and $Z^*\in \mf{Z}(G^*)$. Then $Z^* * f^*$ and $\mf{T}(Z^*)   * f$ are also associated.
\end{theorem} 

\begin{proof}
By \cite[Proposition 2]{Kott}, it suffices to prove matching for regular semisimple orbital integrals, in which case the sign is always +1. When $char( F)=0$ it is shown in \cite[Theorem B.2.c.1]{bdkv1984repr} that for any 
$\varphi\in \Hz(G)$, there exists $\varphi^*\in \Hz(G^*)$ satisfying $O_{\gamma^*}(\varphi^*) = O_{\gamma}(\varphi)$ for corresponding regular semisimple elements, and $O_{\gamma^*}(\varphi^*) =0$ if $\gamma^*$ does not come from $G$. The corresponding result for $F$ having positive characteristic was proved by Badulescu in \cite[Theorem 3.2.]{apoitiers2003resultat}. We apply this to $\varphi=\mf{T}(Z^*)*f$. By Lemma \ref{WIF lemma} we know that $$\tr(\mf{T}(Z^*) * f: \pi)  = (-1)^{n-r} \tr(\varphi^*: JL(\pi))$$ for all tempered irreducible representations $\pi$ of $G$, 
 and similarly since $f$ and $f^*$ are associated, we have $\tr(f:\pi) =(-1)^{n-r} \tr(f^*: JL(\pi))$ for all tempered irreducible representations $\pi$ of $G$. So
 \begin{equation*}
\begin{split}
  \tr(\mf{T}(Z^*) * f: \pi)   &= (\mf{T}(Z^*))(\pi) \tr(f:\pi)  \\
            &= Z^* (JL(\pi) ) (-1)^{n-r}\tr(f^*: JL(\pi)) \\
            &= (-1)^{n-r}\tr(Z^* * f^* :JL(\pi))         
 \end{split}
 \end{equation*}
for all such representations $\pi$. Thus $Z^* * f^*$ and $\varphi^*$ have equal traces on all tempered irreducible representations of $G^*$: those which arise as $JL(\pi)$ for tempered $\pi\in \Pi(G)$, and those which do not (on which both have trace zero by Lemma \ref{WIF lemma}). By Kazhdan's density theorem (\cite{kazhdan1986cuspidal}), $Z^* * f^*$ and $\varphi^*$ have the same regular semisimple orbital integrals.
 \end{proof}

 \begin{cor} \label{matchdistcor}
 Suppose $f$ and $f^*$ are associated. Then $Z_V * f$ and $Z_V^* * f^*$ are also associated.
 \end{cor}
 
This establishes a special case of Conjecture 6.2.2. of \cite{SBC}. The above proof shows that $(-1)^{n-r}\overline{\mf{T}}$ can be used to generate (noncanonical) matching functions. However, as we will see below, this property cannot be lifted to a homomorphism of Hecke algebras.  The proof of this result is similar to the proof of the first part of Theorem C in \cite{scholze2010langlands}, though our functoriality is Jacquet-Langlands rather than Base Change. The second part of that theorem gives a matching statement between characteristic functions of congruence subgroups (in the sense of base change), and is false in our context, and we partially address this in the next section.

\subsection{Explicit Pairs of Matching Functions}\label{sectionmatchingfunction}

In this section we give some important explicit examples of matching functions, to which the previous theorem can then be applied. 

\subsubsection{Euler-Poincare (EP) functions} \label{sectionEPfunctions} Take $f$ and $f^*$ to be Kottwitz's Euler-Poincare functions on the groups $PGL_r(D)$ and $PGL_n(F)$, respectively. 
We can lift them to $G$ and $G^*$, and restrict to the kernels of the respective Kottwitz homomorphisms, $G^1$ and $G^{*,1}$, to get compactly supported functions on $G$ and $G^*$. 
We will simply refer to these as EP functions on $G$ and $G^*$. We make this a little more explicit. Let $K=\GL_n(\Oh_F)$ be the usual maximal compact subgroup of $G^*$, and let $dg^*$ be the Haar measure on $G^*$ such that $dg^*(K)=1$. From \cite[Ch.5]{laumon}, we have the following expression for the EP function on $\GL_n(F)$:
$$f^{EP} = \sum\limits_{I\subset \{1, \ldots, n-1 \}}\left( \frac{(-1)^{n-1-|I|}}{n-|I|} \right)e_{J_I} $$ where $e_{J_I} = \frac{1}{dg^*(J_I)} 1_{J_I} = [K:J_I]1_{J_I}$ and $J_I$ is the standard parahoric defined as follows: if 
$$\{1, \ldots, n-1\}\setminus I = \{d_1, d_1 + d_2, \ldots, d_1 + \ldots +d_{s-1} \}$$ and $\sum\limits_{i=1}^s d_i=n$ then $J_I$ corresponds to the partition $(d_1, \ldots, d_s)$. 

\begin{Remark}
This function is not defined exactly as Kottwitz's original one, but the orbital integrals are unaffected. See \cite[Lemma 5.2.2.]{laumon}.
\end{Remark}

The function $f^{EP}$ depends on a choice of measure $dg^*$, and we are choosing ours so that the maximal compact subgroup has volume 1. Similarly, if $\dim_F D = n^2$, then the EP function on $\GL_1(D)$ is $(vol(D^\times / F^\times, dg/dz))^{-1} 1_{\Oh_D^\times} = \frac{1}{n}1_{\Oh_D^\times}$, where $dg(\Oh_D^\times) =1$ and $dz(\Oh_F^\times)=1$. 

\begin{proposition}\label{EPmatch}
If $\GL_1(D)$ is an inner form of $\GL_n(F)$ then the functions $f^{EP}\in \Hz(G^*, I)$ and $\frac{1}{n} 1_{\Oh_D^\times}\in \Hz(D^\times, \Oh_D^\times)$ have matching orbital integrals. 
\end{proposition}

\begin{proof}
Since the conjugacy classes in $G^*$ that correspond to those in $D^\times$ are precisely the elliptic ones, this proposition is largely a consequence of \cite[Theorem 5.1.3.]{laumon}, which computes the orbital integrals of $f^{EP}$. Since $\Oh_D^\times$ is normal in $D^\times$, computing the orbital integrals of $\frac{1}{n}1_{\Oh_D^\times}$ is trivial: $$O_{\gamma}\left( \frac{1}{n} 1_{\Oh_D^\times}\right) = \begin{cases}
\frac{1}{n}vol (D^\times/ G_\gamma, dg/dg_\gamma) \text{ if } \nu(\gamma)=1 \\
0 \text{                otherwise } 
\end{cases}$$

Comparing with the result of \cite[Theorem 5.1.3]{laumon}, one checks directly that the matching condition holds. \end{proof}

\begin{Remark}
Proposition \ref{EPmatch} can also be regarded as a special case of \cite[Theorem 2]{Kott}, at least when $F$ has characteristic zero. 
\end{Remark}






\subsubsection{Matching for characteristic functions of parahoric subgroups of $\GL_r(D)$} \label{sectionmatchingfor1J} For a parahoric subgroup $J\subset G=\GL_r(D)$, we will construct an explicit function in $\mc{H}(G^*, I^*)$ associated to the characteristic function $1_J\in \Hz(G, J)$. This will generalize Proposition \ref{EPmatch}. To do so we first need to recall a few facts about Deligne-Lusztig functions. To each character $\theta: T\to \C^\times$ of a maximal torus in $T\subset H(\FF_q)$, where $H$ is a connected reductive group over $\FF_q$, we can construct a virtual character $R^H_{T}( \theta)$ of $H(\FF_q)$, by the general method of \cite{deligne1976representations}. For us the ambient group $H$ will be a product of general linear groups (and restrictions of scalars thereof), and we will only need this construction for the case where the character $\theta$ is trivial. 

Recall that conjugacy classes of tori in $\GL_d(\FF_q)$ correspond bijectively to the Weyl group $W=S_d$; we write $T_w$ for any torus in the conjugacy class corresponding to $w\in W$. Write $\overline{G}:=\GL_d(\FF_q)$, and for $w\in S_d$, write $R_w:=R_{T_w}^{\overline{G}}(1)$ for the Deligne-Lusztig function on $\overline{G}$ associated to the trivial character of $T_w$. We will also regard $R_w$ as a function on $K=\GL_d(\Oh_F)$ via inflation, or on $\GL_d(F)$ by extension by zero, since its meaning will always be clear from context.


We consider first the case of $\GL_1(D)$. Here the unique parahoric is $I=\Oh_D^\times$. Let $1_I$ denote the unit element of the Iwahori-Hecke algebra of $\GL_1(D)$, an inner form of $\GL_d(F)$. We now recall a (very) special case Theorem 2.2.6. of \cite{KV}. \begin{theorem}[Kazhdan-Varshavsky] \label{KVtheorem}
Suppose that $F$ is characteristic zero, and suppose that $p>d$, where $p$ is the residue characteristic of $F$. If $w=(12\cdots d)$ then $1_I\in \Hz (\GL_1(D) )$ and $R_w\subset \Hz(\GL_d(F))$ are matching functions. 
 \end{theorem}

    
We will give an alternative proof of Theorem \ref{KVtheorem}, removing these restrictions on the characteristic of $F$, and on $p$. Comparing with the Proposition \ref{EPmatch}, we expect to have a relation between EP and Deligne-Lusztig functions in this context. 

\begin{proposition} \label{EPandKV}
We have, with no restriction on $F$, the equality $$d\cdot O_\gamma(f^{EP}_{\GL_d(F)}) = O_\gamma(R_w)$$ for all $\gamma\in \GL_d(F)$, if $w=(1 2\cdots d)$. 
\end{proposition}
 This will be a consequence of the following Propositions \ref{Comb prop} and \ref{gen prop}. Propositions \ref{EPmatch} and \ref{EPandKV} show that Theorem \ref{KVtheorem} holds with no restriction on the residue characteristic $p$, or on the characteristic of $F$. We will shortly use a similar approach to address the case of general $\GL_r(D)$, though in general we will still have this restriction on $F$. 

\begin{Remark}
The functions $1_I$ and $1_{I^*}$ do not match, as $O_\gamma(1_{I^*})$ both fails to vanish when it should (at nonelliptic regular semisimple $\gamma$), and vanishes when it should not (at elliptic regular $\gamma$). Indeed $1_{I^*}$ does not match any $f\in \Hz(G)$ for this reason. In particular, the normalized transfer homomorphism $\tilde{t}$ above does not generate matching pairs, unlike the base change homomorphism. Nevertheless, Theorem \ref{MatchDist} shows that the analogous result is true at the level of distributions in the Bernstein center. 
\end{Remark} 

We write $K_1$ for the prounipotent radical of $K$, and identify $J_I/K_1$ with a standard parabolic subgroup of block upper-triangular matrices in $K/K_1=\GL_d(\FF_q)$. We recall a few notions and notations from the representation theory of finite groups. If $f$ is a class function on the parabolic subgroup $J_I/K_1$, we write $Ind_{J_I/K_1}^{\GL_d(\FF_q)}(f)$ for the induced class function, defined by $x\mapsto \sum\limits_{s\in S} f(s^{-1}xs)$, where $S$ is a set of left coset representatives for $J_I/K_1$ in $\GL_d(\FF_q)$, and where the summand is by definition zero if $s^{-1}xs\not\in J_I/K_1$.  

\begin{proposition} \label{Comb prop}
Let $y=(12\cdots d)\in S_d=:W$. Then we have
\begin{eqnarray} \label{eqncombprop}
R_y= d\left( \sum\limits_{I\subset \{1, \ldots, d-1  \}} \frac{(-1)^{d-1-|I|}}{d-|I|} Ind_{J_I/K_1}^{\GL_d(\FF_q)} (1)\right).
\end{eqnarray}

\end{proposition}

\begin{proof}

We first describe some notation. Fix a parahoric subgroup $J$ of $\GL_d(F)$, and let $L:=J/J_+$ be its maximal reductive quotient, which is isomorphic to a product of groups of the form $\GL_{d_i}(\FF_q)$. Let $W_L=N_L(T)/T$ for a split maximal torus $T$ of $L$. We will need the formulas 
\begin{equation} \label{eqn1}
1_J = \frac{1}{|W_L|}\sum\limits_{w\in W_L} R_{T_w}^L (1)
\end{equation}
\begin{eqnarray} \label{eqn2}
Ind_{J/K_1}^{\overline{G}} (1) = \frac{1}{|W_L|} \sum\limits_{w\in W_L} R_{T_w}^{\overline{G}}(1)
\end{eqnarray}
 The first is a rewriting of Proposition 7.4.2. in \cite{Carter}, the latter Proposition 7.4.4. of \cite{Carter}. 
We may rewrite the RHS of (\ref{eqncombprop}) using (\ref{eqn2}), and the coefficient of $R_g$, for $g\in W$ is $$f_g := d \sum\limits_{I\subset \{1, \ldots, d-1  \} }\frac{(-1)^{d-1-|I|}}{d-|I|} \frac{1}{|W_I|} |\{v\in W_I: \ v \text{ is conjugate to }  g  \text{ in } S_d  \}|$$ and so it suffices to show that $f_g=0$ unless $g$ is conjugate to $(12\cdots d)$, in which case $f_g=1$. The second assertion is straightforward: $y$ is not in any proper parabolic subgroup $W_I$ of $W=S_d$, so only $I=\{ 1, \ldots d-1\}$ contributes, and $f_y = d\cdot \frac{1}{1} \frac{1}{d!} \cdot |W\cdot y|$ but the conjugacy class of $y$ is size $(d-1)!$ so $f_y=1$. The vanishing property is more difficult.

Let the ``$1$-adic EP function'' be $f=\sum\limits_{I\subset \{1, \ldots, d-1  \}} \frac{(-1)^{d-1-|I|}}{d-|I|} \frac{1}{|W_I|}1_{W_I} $. Then the claim is equivalent to the statement that the orbital integral $O_g(f) :=\sum\limits_{v\in S_d} f(v^{-1}gv) = |C_{S_d}(g) |f_g$ vanishes for $g$ not conjugate to $(1\cdots d)$. Let $W_M\subset W$ be any Young subgroup. Using the notation of \cite{laumon}, we write $D_{M,I}$ for minimal coset representatives for $W_M\backslash W/W_I$. By \cite[Proposition 2.7.5.]{Carter}, every element of the double coset $W_M w W_I$ can be uniquely expressed as $w_M w w_I$ with $w_M\in W_M\cap D_{ \emptyset, J}$, $w_I\in W_I$, and $J:=\Delta^M \cap w(I)$. Then for any $W_M$ (i.e., any Young subgroup) we have 
\begin{equation*}
O_g(1_{W_I}) = \sum\limits_{w\in D_{M,I}} \sum\limits_{w_M\in W_M\cap D_{\emptyset, J}} \sum\limits_{w_I\in W_I} 1_{W_I}(w_I^{-1}w^{-1}w_M^{-1}g w_Mw w_I )$$ 
$$= \sum\limits_{w\in D_{M,I}} \sum\limits_{w_M\in W_M\cap D_{\emptyset, J}} |W_I| 1_{W_I}(w^{-1}w_M^{-1}g w_M w)
\end{equation*}
so 
\begin{equation*}
 O_g(f)=\sum\limits_{I\subset \{1, \ldots, d-1  \}} \sum\limits_{w\in D_{M,I}}\sum\limits_{w_M\in W_M\cap D_{\emptyset, \Delta^M\cap w(I)}}\frac{(-1)^{d-1-|I|}}{d-|I|} 1_{W_I} (w^{-1} w_M^{-1} gw_M w) .
\end{equation*}
If $J=\Delta^M \cap w(I)$ then $W_M\cap wW_I w^{-1}= W_J$. That is, $W_J$ is a parabolic subgroup of $W_M$ whose type is determined by $J$. This is parallel to \cite[Lemma 5.4.6.]{laumon}, which in turn relies on \cite[\S 2.7]{Carter}, especially Theorem 2.7.4. The set $W_J$ is precisely the support of $g\mapsto 1_{W_I}(w^{-1}gw)$, viewed as a function on $W_M$. In other words 
$$1_{W_I}^w|_{W_M} = 1_{W_J}$$ 
for $w\in D_{M,I}$, where $h^w(v) = h(wvw^{-1})$ for any function $h$ on $W$. Noting that $(12\cdots d)$ is the only elliptic conjugacy class in $W$, we may choose $M$ so that $g\in W_M \neq W$, and deduce
$$O_g(f)=\sum\limits_{J\subset \Delta^M} \sum\limits_{\substack{  I\subset \{1, \ldots d-1\}    \\ w\in D_{M, I} : \\ J= \Delta^M\cap w(I)   }  } \sum\limits_{w_M\in W_M\cap D_{\emptyset, J}} \frac{(-1)^{d-1-|I|}}{d-|I|} 1_{W_J}(w_M^{-1}gw_M).$$ 
So it suffices to show that for all $J\subset \Delta^M$ and all $w_M \in W_M \cap D_{\emptyset, J}$, we have $$\sum\limits_{ \substack{  I\subset \{1, \ldots, d-1\} \\  w\in D_{M, I}  \\ J= \Delta^M\cap w(I)  } }   \frac{(-1)^{d-1-|I|}}{d-|I|} 1_{W_J}(w_M^{-1}gw_M)=0. $$ This is obvious if $w_M^{-1}g w_M\not\in W_J$. Otherwise, since $W_M$ is a proper Young subgroup, this is a special case of \cite[Proposition 5.5.5.]{laumon}. We remark that \cite[Proposition 5.5.5.]{laumon} is a purely combinatorial statement, which is used as an important step in demonstrating the vanishing of the nonelliptic orbital integrals for the Euler-Poincare function on $G^*$.
\end{proof}

\begin{proposition} \label{gen prop}
Let $f$ be a class function on the parabolic subgroup $J/K_1$ of $K/K_1 = \GL_n(\FF_q)$, regarded as a function on $J$, where $J \subset K $ is any parahoric subgroup. Then we have $$[K:J] O_\gamma(f) = O_\gamma(Ind_{J/K_1} ^{\GL_n(\FF_q)} (f) )$$ for all $\gamma\in \GL_n(F)$, where $Ind (f)$ is the induced class function.

\end{proposition}

\begin{proof}

Certainly it suffices to consider the case $f=1_C$, the characteristic function of a conjugacy class $C \subset J/K_1$. Then the induced function $K/K_1\to \C$ is
$$Ind_{J/K_1}^{K/K_1}f(s)= \frac{1}{|J/K_1|} \sum\limits_{ \substack{  t\in K/K_1 \\tst^{-1} \in C } }  1 =\sum\limits_{g\in K/J} 1_{gCg^{-1}}(s).$$

The first equality is the definition. The second equality is a simple exercise in finite group theory, but we provide it for completeness. Note that the containment $s\in C^g:=gCg^{-1}$ only depends on the coset $g(J/K_1)$. Further, $s\in C^g$ iff $g^{-1}sg \in C$, so the value of $\sum\limits_{g\in K/K_1} 1_{C^g} = |J/K_1|\sum\limits_{g\in K/J} 1_{C^g}$, which is a class function on $K/K_1$, on $s\in C$, is exactly $|\{t\in K/K_1: tst^{-1}\in C  \}|$, as desired. The conclusion of the proposition is clear, since $O_\gamma(1_{gCg^{-1}})$ is certainly independent of $g$. 
\end{proof}


 \begin{Remark}
In particular, if $f$ is a unipotent character of $K/K_1$, then since it can be written as a linear combination of parabolic inductions of trivial characters, its lift to $K$ has orbital integrals equal to those of some function supported on $K$ which is a linear combination of characteristic functions of parahoric subgroups, and hence is Iwahori-biinvariant. Not all Iwahori-biinvariant functions so arise, since in general the span of $1_J$ as $J$ varies through parahorics $I\subset J \subset K$ is strictly smaller than $\Hz(K, I)$.
 \end{Remark}  

\begin{proof}[Proof of Proposition \ref{EPandKV}]
We apply Proposition \ref{gen prop} to the special case of $f=e_J:=\frac{1}{dg^*(J)}1_J$, using $[K:J]=1/dg^*(J)$ to obtain $$O_\gamma(e_J) = O_\gamma( Ind_{J/K_1}^{K/K_1}(1)).$$ In particular the LHS only depends on the maximal reductive quotient $J/J_+$, which is a weaker invariant than the conjugacy class of $J$. Thus
\begin{equation*}
\begin{split}
O_\gamma(f_{\GL_d(F)}^{EP})  &= O_\gamma\left( \sum\limits_{I\subset \{1, \ldots, d-1  \}} \frac{(-1)^{d-1-|I|}}{d-|I|} Ind_{J_I/K_1}^{\GL_d(\FF_q)} (1)\right)\\
&= \frac{1}{d}O_\gamma(R_{\FF_{q^d}^\times}^{\GL_d(\FF_q)} (1) ) \\ 
\end{split}
\end{equation*}
where the second equality is Proposition \ref{Comb prop}.
\end{proof}

For the Iwahori subgroup in general $\GL_r(D)$ the situation is as follows. We again have the following special case of \cite[Theorem 2.2.6.]{KV}.

\begin{theorem}[Kazhdan-Varshavsky] \label{KVtheorem2}
Suppose that $F$ is characteristic zero, and suppose that $p>n$, where $p$ is the residue characteristic of $F$. If $$w=(12\cdots d)(d+1\cdots 2d)\cdots ( (r-1)d+1\cdots rd  )$$ then $1_I\in \Hz (\GL_r(D) )$ and $R_w=R_{T_r}^{\GL_n(\FF_q)}(1)\in \Hz(\GL_n(F))$ are matching functions. Here $T_r\cong (\FF_{q^{n/r}}^\times)^r$ and $I\subset \GL_r(D)$ is the Iwahori subgroup.
\end{theorem}

 It would be natural to have an Iwahori-biinvariant function on $\GL_n(F)$ which matches $1_I\in \mc{H}(\GL_r(D) )$. Note that a Deligne-Lusztig function is only $K_1$-biinvariant. 
 
 \begin{theorem}
Suppose that $F$ is characteristic zero, and suppose that $p>n$, where $p$ is the residue characteristic of $F$. Define the function 
$$  d^r \sum\limits_{I_1, \ldots, I_r \subset \{ 1, \ldots, d-1 \}} \left( \prod\limits_{i=1}^r \frac{(-1)^{d-1-|I_i|}}{d-|I_i|}  \right) e_{J_{\vec{I}} }$$
where the subscript on the $e$ is understood to mean a parahoric with $\prod\limits_{i=1}^r J_{I_i} / J_{I_{i},+}$ as its maximal reductive quotient. Then this function gives an $I^*$-biinvariant function, supported on the parahoric corresponding to the partition $(d^r)$ of $n$, which is associated to $1_I \in \Hz(G, I)$.

 \end{theorem}

\begin{proof}

To proceed, observe first that by \cite[Prop. 7.4.4.]{Carter},  $$R_{T_r}^{\GL_n(\FF_q)}(1) = Ind_{P_r}^{\GL_n(\FF_q)} \left( R_{\FF_{q^d}^\times}^{\GL_d(\FF_q)} (1)^{\boxtimes r} \right)$$
where $d=n/r$ and $P_r$ is a parabolic with Levi factor $M_r\cong \GL_d(\FF_q)^r$.  
We now apply Proposition \ref{Comb prop}, observing that by transitivity of induction, $$Ind_{P_r}^{\GL_n(\FF_q)} \left(Ind_{Q_1}^{\GL_d(\FF_q)} (1) \boxtimes \cdots \boxtimes Ind_{Q_r}^{\GL_d(\FF_q)}(1)\right) = Ind^{\GL_n(\FF_q)}_{L_1\times \cdots \times L_r} (1) $$ where $L_i$ is any choice of Levi factor of the parabolic subgroup $Q_i\subset \GL_d(\FF_q)$, and on the RHS we mean parabolic induction. Thus 
$$R_{T_r}^{\GL_n(\FF_q)}(1) = d^r \sum\limits_{I_1, \ldots, I_r \subset \{ 1, \ldots, d-1 \}} \left( \prod\limits_{i=1}^r \frac{(-1)^{d-1-|I_i|}}{d-|I_i|}  \right) Ind^{\GL_n(\FF_q)}_{\prod\limits_{i=1}^r J_{I_i} / J_{I_{i},+}}(1)$$
and has orbital integrals equal to that of the function defined in the theorem.
\end{proof}

\begin{Remark}
This is not proportional to a summand of the usual EP function. For example, the coefficient of $e_{I^*}$, which corresponds to taking all subsets to be empty, is $(-1)^{r(d-1)}$, whereas in $f^{EP}$ it is $(-1)^{d-1}\frac{1}{rd}$. However the the term corresponding to taking all subsets to be $\{1, \ldots, d-1\}$ has coefficient $d^r$, whereas in the EP function its coefficient is ($\Delta-I = \{d, 2d, \ldots, (r-1)d  \}$) equal to $(-1)^{r-1}/ r$. 
\end{Remark}

We can also give a more conceptual construction of such a function, similar to Kottwitz's original description of EP functions, which will have the virtue that it applies to any connected reductive $F$-group. Let $M=\GL_d(F)^r\subset \GL_n(F)$. Let $S_M$ be a set of representatives for the $M_{ad}$-orbits (equivalently $M$-orbits) of facets in the Bruhat-Tits building $B(M_{ad})$. For $\sigma\in S_M$ we write $M_\sigma$ for the stabilizer of $\sigma$ in $M$. Then the function
$$f_{S_M} : = d^r\sum\limits_{\sigma\in S_M} \frac{(-1)^{\dim\sigma} }{ vol(M_\sigma/ Z(M), dm/dz )}  \frac{vol(M_\sigma \cap M^1, dm)  }{vol(J_\sigma, dg)}1_{J_\sigma} \in \Hz(G^*)$$ matches $1_I\in \Hz(G)$, where $J_\sigma$ is any parahoric in $\GL_n(F)$ with maximal reductive quotient $J_\sigma/ J_{\sigma, +}$ being isomorphic to the maximal reductive quotient of $M_\sigma \cap M^1$. This has the same orbital integrals as our function above, hence the same as $R_{T_r}^{\GL_n(\FF_q)}(1)$. In fact, the former is an average over certain choices of $S_M$, as is done in \cite[Lemma 5.2.2.]{laumon}, using an identification of the facet $\sigma$ with a product of $r$ facets of $PGL_d(F)$. Observe that if $d=n$ then $f_{S_M}$ is the usual Euler-Poincare function (up to a factor of $d$), and if $d=1$ then it is the characteristic function of the Iwahori $I^*$, divided by $dg^*(I^*)$. Note that the scalar $d^r$ can be interpreted as $|M/ M^1 Z(M)|$, or as the Coxeter number of the Weyl group of $M$. More generally, for $M^*\subset G^*=\GL_n(F)$ an arbitrary Levi subgroup, we define 
$$f_{S_{M^*}} = |M^*/ M^{*,1} Z(M^*)|   \sum\limits_{\sigma\in S_{M^*}} \frac{(-1)^{\dim\sigma} }{ vol(M^*_\sigma/ Z(M^*), dm/dz )}  \frac{vol(M^*_\sigma \cap M^{*,1}, dm)  }{vol(J_\sigma, dg)}1_{J_\sigma}$$ where $J_\sigma$ is any parahoric in $\GL_n(F)$ with maximal reductive quotient $J_\sigma/ J_{\sigma, +}$ being isomorphic to the maximal reductive quotient of $M_\sigma \cap M^1$. We will use these functions shortly. Note that we can choose the $J_\sigma$ in such a way that they all contain a given Iwahori subgroup. We assume such a choice has been made, so that $f_{S_{M^*}}$ is biinvariant with respect to this Iwahori subgroup. 





Finally, let $J$ be an arbitrary parahoric in $\GL_r(D)$. We will now construct an explicit function $F_J\in \Hz(G^*, I^*)$ which matches $1_J\in \Hz(G)$. Let 
$$L := J/J^+ \cong \prod\limits_{i=1}^k \GL_{r_i}(\FF_{q^d})$$ be the maximal reductive quotient of $J$, so that a maximal torus $T(\FF_q)\subset L$ is of the form $T_1(\FF_1)\times \cdots \times T_k(\FF_q)$ where $T_i(\FF_q) \cong \prod\limits_{j=1}^{t_i} \FF^\times_{q^{dr_{ij}}}$ where $\sum\limits_{j=1}^{t_i} r_{ij}=r_i$. 
Write $L^* = \prod\limits_{i=1}^k \GL_{r_i d}(\FF_q)$; this is the maximal reductive quotient of a parahoric subgroup $J^*\subset \GL_n(F)$ corresponding to $J\subset \GL_r(D)$. Note that every maximal torus of $L$ also embeds in $L^*$. If $w\in W_L$, we will write $T_w$ for the corresponding (conjugacy class of) maximal torus, which we will regard as being embedded in $L$ or $L^*$ as needed. 
Then $R_{T_w}^{L}(1)$ lifts to a function on $J$, and we have $1_J= \frac{1}{|W_{L}|} \sum\limits_{w\in W_L} R_{T_w}^{L}(1)$ as above. By a special case of \cite[Theorem 2.2.6.]{KV}, this then matches $$\frac{1}{|W_L|} \sum\limits_{w\in W_L} R_{T_w}^{\GL_n(\FF_q)}(1) =\frac{1}{|W_L|} \sum\limits_{w\in W_L} Ind_{L^*}^{\GL_n(\FF_q)} (R_{T_w}^{L^*}(1)).$$ 
We can repeat the analysis above for the Iwahori for each summand $Ind_{L^*}^{\GL_n(\FF_q)} (R_{T_w}^{L^*}(1)) $ to obtain an $f^*\in \Hz(G^*, I^*)$ associated to $1_J$. Explicitly, if to $w\in W_L$, or more properly $T_w$, we associate integers $r_{ij}$ and $t_i$ as above (we suppress dependence of these on $w$), then an Iwahori-biinvariant matching function to $1_I$ is 
$$\frac{1}{|W_L|} \sum\limits_{w\in W_L} \sum\limits_{\substack{   \vec{I}= \{I_{ij} \}  \\  I_{ij}\subset \{1, \ldots, dr_{ij}-1\}  \\  }    }  \left(\prod\limits_ {  \substack {  1\leq i \leq k \\  1\leq j\leq t_i  }  }  \frac{  dr_{ij} (-1)^{ dr_{ij} -1-|I_{ij}| }    }{ dr_{ij} - |I_{ij}|   } \right) e_{J_{\vec{I} } }$$ where again $J_{\vec{I} }$ means a parahoric with maximal reductive quotient isomorphic to $\prod\limits_{i,j } J_{I_{ij} } / J_{I_{ij} }^+  $. Once again we can describe this more conceptually. Let ${\bf T}_w$ be an unramified torus in $G=\GL_r(D)$ defined over $\Oh_{F}$ so that ${\bf T }_w (\FF_q) = T_w$. Then ${\bf T}_w $ also embeds in $G^*=\GL_n(F)$. Let $M_w^* \subset G^*$ be a Levi subgroup such that ${\bf T}_w(F) \subset M_w^*$ and ${\bf T }_w (F)$ is elliptic in $M_w^*$. This latter property determines the conjugacy class of $M_w$. By identical reasoning to the Iwahori case, we now obtain the following result. 

\begin{theorem} \label{defofF_J}
Suppose that $F$ is characteristic zero and $p>n$. Let $J\subset G=\GL_r(D)$ be a parahoric subgroup, and let $W_L$ be the Weyl group of the maximal reductive quotient $L$ of $J$. 
Then $$F_J:=\frac{1}{|W_L|} \sum\limits_{w\in W_L}  f_{S_{M_w^* }} $$ is an Iwahori-biinvariant function associated to $1_J$. For any algebraic, finite-dimensional representation $V$ of ${^L} G$, the functions $Z_V * 1_J$ and $Z_V^* * F_J$ are associated. 
\end{theorem}





\begin{Remark}
Kazhdan and Varshavsky (\cite{KV}) prove that the DL functions have the correct orbital integrals, subject to assumptions on the residue characteristic being sufficiently large. It would be desirable to have a proof of this that has no such constraint.  
\end{Remark}


\begin{Remark}
Once we have the matching functions for $1_J$, we can apply Corollary \ref{matchdistcor}, and since every element of $Z(G, J)$ is of the form $Z_V * 1_J$, as $V$ ranges over representations of the dual group of $G$, we obtain an explicit Iwahori-biinvariant matching function 
for every element of $Z(G, J)$. 
\end{Remark} 


We wish to formulate a (naive) conjecture for more general reductive groups. Let $G$ be an arbitrary reductive group, and $G^*$ be its quasisplit inner form. Let $J\subset G$ be an arbitrary parahoric subgroup. To the function $1_J$ we wish to associate an Iwahori-biinvariant matching function on $G^*$. 

\begin{definition} \label{gendef}
Let $M$ be a Levi subgroup of the connected reductive group $G$, let $S_M$ be a choice of representatives for the $M$-orbits of facets in the Bruhat-Tits building $B(M_{ad})$. The define the ``relative Euler-Poincare function'' to be 
$$f_{S_M} := \left|\frac{M}{M^1 Z(M)}\right| \sum\limits_{\sigma\in S_M}\frac{(-1)^{\dim\sigma} }{ vol(M_\sigma/ Z(M), dm/dz )}  \frac{vol(M_\sigma \cap M^1, dm)  }{vol(J_\sigma, dg)}1_{J_\sigma}$$ where $M_\sigma$ is the stabilizer of $\sigma$, and $J_\sigma$ is a parahoric subgroup of $G$ such that $J_\sigma\cap M = M_\sigma\cap M^1$ and $J_\sigma$ is minimal for this property. 
\end{definition}

\begin{lemma}
The orbital integrals of $f_{S_M}$ are independent of the choices of $S_M$ and the $J_\sigma$.  
\end{lemma}

\begin{proof}
Suppose that $J'_\sigma$ was a different choice. We only need to remark that the condition in the definition pins down the group  $J_\sigma/J_\sigma^+$, since it is this group that determines the orbital integrals of $1_{J_\sigma}$, as follows from (the natural generalization of) Proposition \ref{gen prop}. The independence from $S_M$ is clear since changing it is equivalent to conjugating various summands of the function.  
\end{proof}

Let $L=J/J^+$, so that to $w\in W_L$ we may associate a (conjugacy class of) maximal torus $T_w$ in $L$. If $T_w$ is a maximal torus in $L$ then there is an unramified maximal torus ${\bf T}_w$ defined over $\Oh$, with special fiber $T_w$, and embedding into $G$ and into $G^*$. Let $M^*_w$ be a Levi subgroup of $G^*$ containing ${\bf T}_w (F)$ as an elliptic subtorus. 

\begin{conj}\label{conj}
If $J\subset G$ is an arbitrary parahoric subgroup, then a matching function (in the sense of standard endoscopy) for $1_{J}$ is $\frac{1}{|W_L|}\sum\limits_{w\in W_L} f_{S_{M_w^*}  }  $
\end{conj} 

In the special case where $G=G^*$, if $J=I$ is the Iwahori, then the function in conjecture \ref{conj} simplifies to $\frac{1}{dg(I)}1_I$. This does match $1_I$ for our definition of matching.



\bibliographystyle{plain}
\bibliography{bibfile}

\end{document}